\documentclass[12pt,oneside]{amsart}

\usepackage[a4paper]{geometry}

\usepackage{amsmath}
\usepackage{amsthm}
\usepackage{amsfonts}
\usepackage{amssymb}
\usepackage{mathabx}
\usepackage{mathtools}

\usepackage{graphicx}
\usepackage{epstopdf}
\usepackage{caption}
\usepackage{subcaption}
\usepackage[nodayofweek]{datetime}
\usepackage{hyperref}
\usepackage{cleveref}
\usepackage{hyphenat}

\newcommand{\sminus}{\mathbin{\setminus}}
\newcommand{\pitch}{\mathbin{\pitchfork}}

\newtheorem{theorem}{Theorem}[section]
\newtheorem{lemma}[theorem]{Lemma}
\newtheorem{proposition}[theorem]{Proposition}
\newtheorem{corollary}[theorem]{Corollary}

\theoremstyle{definition}
\newtheorem{definition}[theorem]{Definition}
\newtheorem{claim}[theorem]{Claim}
\newtheorem{example}[theorem]{Example}

\newcommand{\qedclaim}{\addvspace{\topsep}}
\newcommand{\newcase}{\addvspace{.6\topsep}}

\newcommand{\sep}{\text{Sep}}
\newcommand{\C}{\mathcal{C}}
\newcommand{\ksep}{\mathcal{K}}

\newcommand{\mcg}{\text{MCG}}

\newcommand{\X}{\mathfrak{X}}
\newcommand{\nest}{\sqsubseteq}
\newcommand{\nestneq}{\sqsubsetneq}
\newcommand{\G}{\mathcal{G}}

\newcommand{\W}[1]{{W}^{({#1})}}
\newcommand{\N}[2]{N^{(#1)}_{#2}}
\newcommand{\diam}{\text{diam}}

\newcommand{\Gam}[1]{{\Gamma}^{({#1})}}
\newcommand{\mygraph}{twist\hyp{}free multicurve graph}

\newcommand{\A}{\mathcal{A}}

\makeatletter
\def\subsection{\@startsection{subsection}{2}
  \z@{.5\linespacing\@plus.7\linespacing}{.3\linespacing}
  {\normalfont\scshape}}
\makeatother

\title{Hierarchical hyperbolicity of graphs of multicurves}
\author{Kate M. Vokes}
\address{Institut des Hautes \'Etudes Scientifiques, 35 route de Chartres, 91440 Bures-sur-Yvette, France}
\email{vokes@ihes.fr}

\thanks{Date: 2 August 2019}

\begin{document}

\maketitle

\begin{abstract}
We show that many graphs naturally associated to a connected, compact, orientable surface are hierarchically hyperbolic spaces in the sense of Behrstock, Hagen and Sisto.
They also automatically have the coarse median property defined by Bowditch.
Consequences for such graphs include a distance formula analogous to Masur and Minsky's distance formula for the mapping class group, an upper bound on the maximal dimension of quasiflats, and the existence of a quadratic isoperimetric inequality.
The hierarchically hyperbolic structure also gives rise to a simple criterion for when such graphs are Gromov hyperbolic.
\end{abstract}

\section{Introduction} \label{section: introduction}

Let $S$ be a connected, compact, orientable surface.
Over the past decades, various graphs and complexes have been defined where each vertex represents an isotopy class of curves or multicurves in~$S$.
Such graphs have proved an important tool in the study of the large scale geometry of mapping class groups, Teichm\"uller theory and the geometry of hyperbolic 3\hyp{}manifolds.
A first example is the curve graph, defined by Harvey~\cite{harvey}, which has a vertex for every isotopy class of curves in the surface, with an edge joining two vertices if the corresponding curves can be realised disjointly.
We equip this graph with the combinatorial metric $d_S$ defined by setting each edge to have length~1.
Masur and Minsky proved that the curve graph is Gromov hyperbolic, with infinite diameter~\cite{mm1}.
Moreover, in~\cite{mm2}, they gave a distance formula for the mapping class group, proving that distances in the word metric can be approximated in terms of a sum of projections to curve graphs of subsurfaces.
In the time since this result, this distance formula has been generalised to many other spaces associated to surfaces; see for example~\cite{mahierarchy, ms, mjrank, rafi07, sultan}.

One generalisation of the results of~\cite{mm2} is the notion of a hierarchically hyperbolic space, defined by Behrstock, Hagen and Sisto~\cite{hhs1, hhs2}.
This property in particular implies the existence of a distance formula analogous to that for the mapping class group.
The idea is to state necessary conditions, based on results from~\cite{mm2} and elsewhere, to give consequences such as the distance formula.
In particular, every hierarchically hyperbolic space is equipped with projections to a family of hyperbolic spaces, by analogy with subsurface projections to curve graphs.
We shall state the definition fully in Section~\ref{subsection: hhs}.

Hierarchical hyperbolicity also implies the coarse median property defined by Bowditch in~\cite{bowmedian}.
This is a notion of non\hyp{}positive curvature for which mapping class groups are again a motivating example.
A coarse median space is equipped with a ternary operator which is approximated on finite subsets by the median operation on a finite median algebra.

In this paper, we show that graphs of multicurves associated to surfaces, satisfying certain natural conditions, are hierarchically hyperbolic spaces, and also derive some consequences.
For some of the graphs to which our results apply, such as the pants graph, hierarchical hyperbolicity is already known, and for others at least some of the consequences stated below in Section~\ref{subsection: results} are already understood.
However, our results cover a fairly general family of graphs associated to surfaces and we are able to deduce new information about interesting examples, such as the separating curve graph.

Our result is also applied in recent work of Russell~\cite{russell} to prove that certain of the graphs we consider are relatively hyperbolic.
This applies, for example, to separating curve graphs of closed surfaces.

\subsection{Statement of results} \label{subsection: results}

We will call the graphs to which our results apply \mygraph{}s.
We will give a full definition of this in Section~\ref{subsection: mygraph definition}, along with some examples.
Also in Section~\ref{section: preliminaries}, we will give more background on curve graphs and subsurface projections and on hierarchically hyperbolic spaces.

\begin{theorem} \label{theorem: g hhs}
Let $S$ be a surface and $\G(S)$ a \mygraph{}.
Let $\X$ be the set of subsurfaces such that for every $X \in \X$, every vertex of $\G(S)$ has non\hyp{}trivial subsurface projection to~$X$.
Then $\G(S)$ is a hierarchically hyperbolic space with respect to subsurface projections to the curve graphs of subsurfaces in~$\X$.
\end{theorem}

Corollary \ref{corollary: g distance formula} below is a distance formula for~$\G(S)$, analogous to that of Masur and Minsky for the mapping class group.
It follows immediately from Theorem~\ref{theorem: g hhs}, using \cite[Theorem~4.5]{hhs2}.
Here, the notation $A \asymp_{K_1,K_2} B$ means $\frac{1}{K_1}(A-K_2) \le B \le K_1A+K_2$.
The function $\left[\ \right]_C$ is the cutoff function where $\left[x\right]_C=x$ when $x\ge C$ and $\left[x\right]_C=0$ when $x<C$.
The map $\pi_X$ is the subsurface projection from $\G(S)$ to~$\C(X)$ (see Section \ref{subsection: subsurface projection}).

\begin{corollary} \label{corollary: g distance formula}
Let $\G(S)$ be a \mygraph{}.
Then there exists a constant $C_0$ such that for every $C \ge C_0$ there exist $K_1$ and $K_2$ such that the following holds.
For every pair $a$, $b$ of vertices of $\G(S)$, we have:
\[ \pushQED{\qed} d_{\G(S)}(a, b) \asymp_{K_1,K_2} \sum_{X \in \X} \left[d_X(\pi_X(a), \pi_X(b)) \right]_C. \qedhere \popQED \]
\end{corollary}

\begin{corollary} \label{corollary: g coarse median}
Let $\G(S)$ be a \mygraph{} and let $\nu$ be the maximal cardinality of a set of pairwise disjoint subsurfaces in $\X$.
Then $\G(S)$ is a coarse median space of rank $\nu$.
\end{corollary}

\begin{proof}
This will follow from \cite[Corollary~2.15]{hhsflats}.
The relation of orthogonality here corresponds exactly to disjointness of subsurfaces; more details will be given later in Section~\ref{subsection: 1 to 8}.
Since the image of $\G(S)$ in each $\C(X)$ has infinite diameter, the rank of~$\G(S)$, as defined in \cite[Definition~1.9]{hhsflats}, is the maximal cardinality of a set of pairwise orthogonal elements of~$\X$, which is exactly~$\nu$.
For the same reason, the condition of being ``asymphoric'' in the sense of \cite[Definition~1.13]{hhsflats} is satisfied.
Hence the conclusion of \cite[Corollary~2.15]{hhsflats} is satisfied, giving the required result.
\end{proof}

Corollaries \ref{corollary: g rank} and \ref{corollary: g hyp} follow by \cite[Theorem~1.14]{hhsflats} (see also \cite[Lemma~6.10]{bowmcg}) and \cite[Theorem~2.1]{bowmedian} respectively. 

\begin{corollary} \label{corollary: g rank}
Let $\G(S)$ be a \mygraph{} and let $\nu$ be the maximal cardinality of a set of pairwise disjoint subsurfaces in~$\X$. 
Then the maximal $n$ such that for some fixed $K>0$ and for every $R>0$, there exists a $(K,K)$\hyp{}quasi\hyp{}isometric embedding into $\G(S)$ of the Euclidean ball of dimension $n$ and radius $R$ is \mbox{$n=\nu$}. 
\qed
\end{corollary}

\begin{corollary} \label{corollary: g hyp}
Let $\G(S)$ be a \mygraph{}.
Suppose that there exists no pair of disjoint subsurfaces in the set $\X$.
Then $\G(S)$ is Gromov hyperbolic. \qed
\end{corollary}

\begin{corollary} \label{corollary: g isoperimetric}
Let $\G(S)$ be a \mygraph{}.
Then $\G(S)$ satisfies a quadratic isoperimetric inequality in the sense of \cite[Proposition~8.2]{bowmedian}. \qed 
\end{corollary}

\subsubsection*{Acknowledgements.} 
I am grateful to my PhD supervisor, Brian Bowditch, for many invaluable suggestions and interesting conversations, and for thorough comments on earlier versions of this paper. 
I would like to thank Saul Schleimer for helpful discussions and comments, and particularly for suggesting an alternative to the original proof of Lemma~\ref{lemma: coarse lipschitz retract} which facilitated generalising from the original case.
I would also like to thank Jacob Russell for many interesting conversations and comments on this paper, Bert Wiest for discussions on applying the results of this paper to the arc graph (Appendix \ref{appendix: arc graph hhs}), and Kasra Rafi and Henry Wilton for helpful feedback. 
Most of this research was carried out at the University of Warwick, supported by an Engineering and Physical Sciences Research Council Doctoral Award.
Much of the writing was completed at the Fields Institute for Research in Mathematical Sciences, supported by a Fields Postdoctoral Fellowship.

\section{Preliminaries} \label{section: preliminaries}

In this section, we give some background and state the definition of a hierarchically hyperbolic space.

\subsection{Curves and subsurface projection} \label{subsection: subsurface projection}

We say that a simple closed curve in a surface~$S$ is \emph{essential} if it is not homotopic to a point and \emph{non\hyp{}peripheral} if it is not homotopic to a boundary component of~$S$.
In this paper, any curves will be essential, non\hyp{}peripheral simple closed curves.

A multicurve in $S$ is a collection of pairwise disjoint, pairwise non\hyp{}isotopic curves.
Two multicurves $a$ and $b$ are in \emph{minimal position} if the number of intersections between $a$ and $b$ is minimal among all pairs of multicurves $a'$, $b'$ isotopic to $a$, $b$ respectively.
The \emph{intersection number}, $i(a, b)$, of two multicurves $a$ and $b$ is the number of intersections between $a$ and $b$ when they are realised in minimal position.
Unless otherwise stated, we will be considering curves and multicurves up to isotopy.

The \emph{mapping class group}, $\mcg(S)$, of $S$ is the group of isotopy classes of orientation preserving homeomorphisms fixing the boundary of $S$ pointwise (where the isotopies must also fix the boundary pointwise).

We shall be considering several graphs associated to a surface $S$ which have curves or multicurves as vertices.
For notational convenience, we shall usually consider these as discrete sets of vertices with the combinatorial metric induced from the graphs.
Maps between the graphs should be considered as maps between their vertex sets and will not necessarily be graph morphisms.
The importance of connectedness for the graphs we will be considering is the consequence that the distance between any two vertices is finite.

As already stated in the introduction, for $\xi(S) \ge 2$, the curve graph, $\C(S)$, has a vertex for every isotopy class of curves, with an edge joining two distinct vertices whenever they have disjoint representatives.
When $\xi(S) = 1$, there are no pairs of disjoint curves on $S$, so we modify the definition so that there is an edge between two vertices whenever the corresponding curves intersect minimally (once for $S_{1,0}$ and $S_{1,1}$ and twice for~$S_{0,4}$).
When $S=S_{0,3}$, the curve graph is empty.
However, a curve graph (or more precisely, arc graph) is defined for~$S_{0,2}$.
The vertex set is the set of arcs in $S_{0,2}$ joining the two boundary components, up to isotopy fixing the boundary.
Two vertices are connected by an edge whenever the arcs have disjoint interiors.
Hence this graph coarsely measures twists about the core curve of the annulus, and is in fact quasi\hyp{}isometric to $\mathbb{Z}$ (see~\cite[Section~2.4]{mm2}).

An \emph{essential} subsurface of a surface $S$ is a connected subsurface $X$ so that every boundary component of $X$ is either a boundary component of $S$ or an essential, non\hyp{}peripheral curve of $S$.
From now on, the word ``subsurface'' will always refer to an isotopy class of essential subsurfaces.
The \emph{complexity}, $\xi(S)$, of a surface $S=S_{g,b}$ is defined by $\xi(S)=3g+b-3$.
This is the maximal number of curves in a multicurve of $S$, and is strictly decreasing under taking proper subsurfaces.
Given a subsurface $X$ of $S$, we denote by $\partial_S X$ the multicurve of $S$ made up of the boundary components of $X$ which are not in $\partial S$.

Given a surface $S$ and a subsurface $X$ of $S$, we have a \emph{subsurface projection} map $\pi_X$ from $\C(S)$ to the power set $2^{\C(X)}$ of $\C(X)$ (in particular, the image of a point under this map could be empty).
As mentioned above, we here think of curve graphs and similar graphs as discrete sets of vertices.
We briefly recall the definition of the subsurface projection map from \cite[Section~2]{mm2}.

Let $X$ be a subsurface of $S$ of positive complexity, and $\alpha$ a curve realised in minimal position with~$\partial_S X$.
If $\alpha$ is contained in $X$ then $\pi_X(\alpha)=\alpha$, and if $\alpha$ is disjoint from (or peripheral in) $X$ then $\pi_X(\alpha)=\varnothing$.
Otherwise, for each arc $\delta$ of intersection of $\alpha$ with~$X$, we take the boundary components of a small regular neighbourhood of $\delta \cup \partial_S X$ which are non\hyp{}peripheral in~$X$.
The union of these curves over all such $\delta$ is~$\pi_X(\alpha)$.

We may similarly consider a subsurface projection from $\mathcal{G}(S)$ to $\C(X)$ for any complex $\mathcal{G}(S)$ whose vertices are curves or multicurves in $S$, and any subsurface $X$ of~$S$.
The projection of a multicurve is the union of the projections of its component curves.
Again, this is a map to the power set~$2^{\C(X)}$.
However, by \cite[Lemma~2.3]{mm2}, if $X$ is a subsurface of $S$ of positive complexity, and $a$ is a multicurve with non\hyp{}empty subsurface projection to $X$, then $\text{diam}_{\C(X)}(\pi_X(a)) \le 2$.
We define the distance between two sets $C$, $D$ of curves in $X$ by $d_X(C, D)=\text{diam}_{\C(X)}(C \cup D)$.
We usually abbreviate $d_X(\pi_X(A), \pi_X(B))$ by $d_X(A, B)$.

We will not use any details of the subsurface projection to an annulus here.
Recall, however, that, as before, the subsurface projection of a multicurve $c$ to an annulus $A$ is non\hyp{}empty if and only if $c$ cannot be isotoped to be disjoint from~$A$.
In particular, the projection to $A$ of its core curve is empty.

Given a complex $\mathcal{G}(S)$, the subsurfaces of $S$ to which every vertex of $\mathcal{G}(S)$ has non\hyp{}trivial subsurface projection are of particular interest.
We call these subsurfaces \emph{witnesses} for~$\G(S)$.
Notice that except for the case of~$S_{0,3}$,  a multicurve $a$ having non\hyp{}trivial subsurface projection to $X$ is equivalent to the statement that $a$ intersects $X$ non\hyp{}trivially (that is, $a$ cannot be isotoped to be disjoint from~$X$).
However, since the curve graph of $S_{0,3}$ is empty, a subsurface $X \cong S_{0,3}$ cannot be a witness even if every vertex of $\G(S)$ intersects it non\hyp{}trivially.

\subsection{Hierarchically hyperbolic spaces} \label{subsection: hhs}

Hierarchically hyperbolic spaces were defined by Behrstock, Hagen and Sisto in~\cite{hhs1}.
Hierarchical hyperbolicity of a space $\Lambda$ is always with respect to some family of uniformly hyperbolic spaces, with projections from $\Lambda$ to these spaces.
The same authors give an equivalent definition of hierarchically hyperbolic spaces in~\cite{hhs2}, and that is the definition we shall use here.
For an exposition of the topic of hierarchically hyperbolic spaces, see~\cite{hhssurvey}.
The space $\Lambda$ is assumed to be a \emph{quasigeodesic space}, that is, any two points in the space can be connected by a quasigeodesic with uniform constants.
All of the spaces we will deal with in this paper will in fact be geodesic spaces.

We say that $(\Lambda, d_\Lambda)$ is a \emph{hierarchically hyperbolic space} if there exist a constant $\delta \ge 0$, an indexing set $\mathfrak{S}$ and, for each $X \in \mathfrak{S}$, a $\delta$\hyp{}hyperbolic space $(\C(X), d_X)$ such that the following axioms are satisfied.

\textbf{1. Projections.} There exist constants $c$ and $K$ such that for each $X \in \mathfrak{S}$, there is a $(K, K)$\hyp{}coarsely Lipschitz \emph{projection} $\pi_X\colon \Lambda \rightarrow 2^{\C(X)} \sminus \varnothing$ such that the image of each point of~$\Lambda$ has diameter at most $c$ in~$\C(X)$.
Moreover, for each $X \in \mathfrak{S}$, $\pi_X(\Lambda)$ is $K$\hyp{}quasiconvex in~$\C(X)$.

\textbf{2. Nesting.} The set $\mathfrak{S}$ has a partial order $\nest$, and if $\mathfrak{S}$ is non-empty then it contains a unique $\nest$\hyp{}maximal element.
If $X \nest Y$ then we say that $X$ is \emph{nested} in~$Y$.
For all $X \in \mathfrak{S}$, $X \nest X$.
For all $X, Y \in \mathfrak{S}$ such that $X \nestneq Y$ (that is, $X \nest Y$ and $X \ne Y$) there is an associated non\hyp{}empty subset $\pi_Y(X) \subseteq \C(Y)$ with diameter at most~$c$, and a projection map $\pi^Y_X\colon \C(Y) \rightarrow 2^{\C(X)}$.

\textbf{3. Orthogonality.}
There is a symmetric and anti-reflexive relation $\perp$ on $\mathfrak{S}$ called \emph{orthogonality}.
Whenever $X \nest Y$ and $Y \perp Z$, $X \perp Z$.
For every $X \in \mathfrak{S}$ and $X \nest Y$, either there is no $U \nest Y$ such that $U \perp X$, or there exists $Z \nestneq Y$ such that whenever $U \nest Y$ and $U \perp X$, $U \nest Z$.
If $X \perp Y$ then $X$ and $Y$ are not $\nest$\hyp{}comparable, that is, neither is nested in the other.

\textbf{4. Transversality and consistency.}
If $X$ and $Y$ are not orthogonal and neither is nested in the other, then we say $X$ and $Y$ are \emph{transverse}, $X \pitch Y$.
There exists $\kappa \ge 0$ such that whenever $X \pitch Y$ there are non\hyp{}empty sets $\pi_X(Y) \subseteq \C(X)$ and $\pi_Y(X) \subseteq \C(Y)$, each of diameter at most $c$, satisfying, for all $a \in \Lambda$:
\[ \min \lbrace d_X(\pi_X(a), \pi_X(Y)), d_Y(\pi_Y(a), \pi_Y(X)) \rbrace \le \kappa. \]
If $X \nest Y$ and $a \in \Lambda$ then:
\[ \min \lbrace d_Y(\pi_Y(a), \pi_Y(X)), \text{diam}_{\C(X)}(\pi_X(a) \cup \pi^Y_X(\pi_Y(a))) \rbrace \le \kappa. \]
These are called the \emph{consistency inequalities}.
If $X \nest Y$, then for any $Z \in \mathfrak{S}$ such that each of $X$ and $Y$ is either transverse to $Z$ or strictly nested in $Z$, we have $d_Z(\pi_Z(X), \pi_Z(Y)) \le \kappa$.

\textbf{5. Finite complexity.}
There exists $n \ge 0$, called the \emph{complexity} of $\Lambda$ with respect to $\mathfrak{S}$, such that any set of pairwise $\nest$\hyp{}comparable elements of $\mathfrak{S}$ contains at most $n$ elements.

\textbf{6. Large links.}
There exist $\lambda \ge 1$ and $E \ge \max \lbrace c, \kappa \rbrace$ such that the following holds.
Let $X \in \mathfrak{S}$, $a, b \in \Lambda$ and $R = \lambda d_X(\pi_X(a), \pi_X(b)) + \lambda$.
Then either $d_Y(\pi_Y(a), \pi_Y(b)) \le E$ for every $Y \nestneq X$, or there exist $Y_1, \dots, Y_{\lfloor R \rfloor}$ such that for each $1 \le i \le \lfloor R \rfloor$, $Y_i \nestneq X$, and such that for all $Y \nestneq X$, either $Y \nest Y_i$ for some $i$, or $d_Y(\pi_Y(a), \pi_Y(b)) \le E$.
Also, $d_X(\pi_X(a), \pi_X(Y_i)) \le R$ for each \nolinebreak $i$.

\textbf{7. Bounded geodesic image.}
For all $X \in \mathfrak{S}$, and $Y \nestneq X$, and for all geodesics $g$ of $\C(X)$, either $\text{diam}_{\C(Y)}(\pi^X_Y(g)) \le E$ or $g \cap N_{\C(X)}(\pi_X(Y), E) \ne \varnothing$.

\textbf{8. Partial realisation.}
There exists a constant $r$ with the following property.
Let $\lbrace X_j \rbrace$ be a set of pairwise orthogonal elements of $\mathfrak{S}$ and let $\gamma_j \in \pi_{X_j}(\Lambda) \subseteq \C(X_j)$ for each $j$.
Then there exists $a \in \Lambda$ such that:
\begin{itemize}
\item $d_{X_j}(\pi_{X_j}(a), \gamma_j) \le r$ for all $j$,
\item for each $j$ and each $X \in \mathfrak{S}$ such that $X_j \nest X$, $d_X(\pi_X(a), \pi_X(X_j)) \le r$,
\item if $Y \pitch X_j$ for some $j$, then $d_Y(\pi_Y(a), \pi_Y(X_j)) \le r$.
\end{itemize}

\textbf{9. Uniqueness.}
For all $K \ge 0$, there exists $K'$ such that if $a, b \in \Lambda$ satisfy \linebreak[4] $d_{X}(\pi_X(a), \pi_X(b)) \le K$ for all $X \in \mathfrak{S}$, then $d_{\Lambda}(a, b) \le K'$.

\subsection{Definition of \mygraph} \label{subsection: mygraph definition}

We now specify those graphs to which our results will apply.

\begin{definition} \label{definition: mygraph}
Let $S$ be a connected, compact, orientable surface.
A graph $\G(S)$ associated to $S$, with the combinatorial metric, is a \emph{\mygraph{}} if it satisfies the following properties.
\begin{enumerate}
\item The graph $\G(S)$ is connected.
\item Each vertex of $\G(S)$ represents a multicurve in $S$.
\item The action of $\mcg(S)$ on the surface $S$ induces an isometric action of $\mcg(S)$ on $\G(S)$.
\item There exists $R$ such that for any pair of adjacent vertices $a$, $b$ of $\G(S)$, $i(a, b) \le R$.
\item The set of witnesses for $\G(S)$ does not contain annuli.
\end{enumerate}
\end{definition}

We now give some examples of graphs associated to surfaces which satisfy these conditions.
Note that the set $\X$ referred to in Theorem~\ref{theorem: g hhs} is the set of witnesses for~$\G(S)$.

\begin{example}
The curve graph, $\C(S)$, is a \mygraph{} for every surface of positive complexity.
The only witness for $\C(S)$ is $S$ itself.
The subsurface projection from $\C(S)$ to itself is the identity map, and this gives the trivial hierarchically hyperbolic structure which results from the hyperbolicity of $\C(S)$.
Hence the conclusion of Theorem~\ref{theorem: g hhs} is nothing new in this example.
\end{example}

\begin{example}
The pants graph, $\mathcal{P}(S)$, is a \mygraph{} for every surface of positive complexity.
The set of witnesses is the set of all positive complexity subsurfaces.
The fact that this gives a hierarchically hyperbolic structure on $\mathcal{P}(S)$ is noted in \cite[Theorem~G]{hhs1}.
\end{example}

\begin{example}
The separating curve graph, $\sep(S)$, is the full subgraph of $\C(S)$ spanned by separating curves, whenever this is connected.
In the cases of $S_{1,2}$, $S_{2,0}$ and $S_{2,1}$, this subgraph of the curve graph is non\hyp{}empty but disconnected.
However, it is standard to modify the definition so that two curves are adjacent whenever their intersection number is minimal among all pairs of distinct separating curves, and this does give connected graphs.
We then have that $\sep(S)$ is a \mygraph{} whenever it is non\hyp{}empty.
Notice that a subsurface of $S$ does not contain any separating curve precisely when it has genus~0 and contains at most one boundary component of~$S$.
Hence, a subsurface $X$ of $S$ is a witness for $\sep(S)$ whenever every component of the complement of $X$ is a planar subsurface containing at most one curve of $\partial S$.
This in particular means that the possibility for two witnesses to be disjoint is very restricted.
When $S$ has at least three boundary components, there are no pairs of disjoint witnesses for $\sep(S)$, and so by Corollary~\ref{corollary: g hyp}, $\sep(S)$ is hyperbolic.
When $S=S_{g,b}$ with $g\ge3$ and $b\le2$, and when $S=S_{2,2}$, there exist pairs of disjoint witnesses (see Figure~\ref{figure: disjointholes}).
However, there is no triple of pairwise disjoint witnesses.
Hence, in this case $\sep(S)$ has rank~2 in the sense of Corollary~\ref{corollary: g rank}.
In~\cite{russell}, this hierarchically hyperbolic structure on $\sep(S)$ is used to prove the relative hyperbolicity of $\sep(S_{g,0})$ when $g\ge3$ and $\sep(S_{g,2})$ when~$g\ge2$.
A more detailed study of the geometry of the separating curve graph in the various cases will appear in forthcoming work with Russell~\cite{russellvokes}.
\end{example}

\begin{figure}[h!]
\centering
\begin{subfigure}[b]{0.22\textwidth}
\centering
\includegraphics[width=\textwidth]{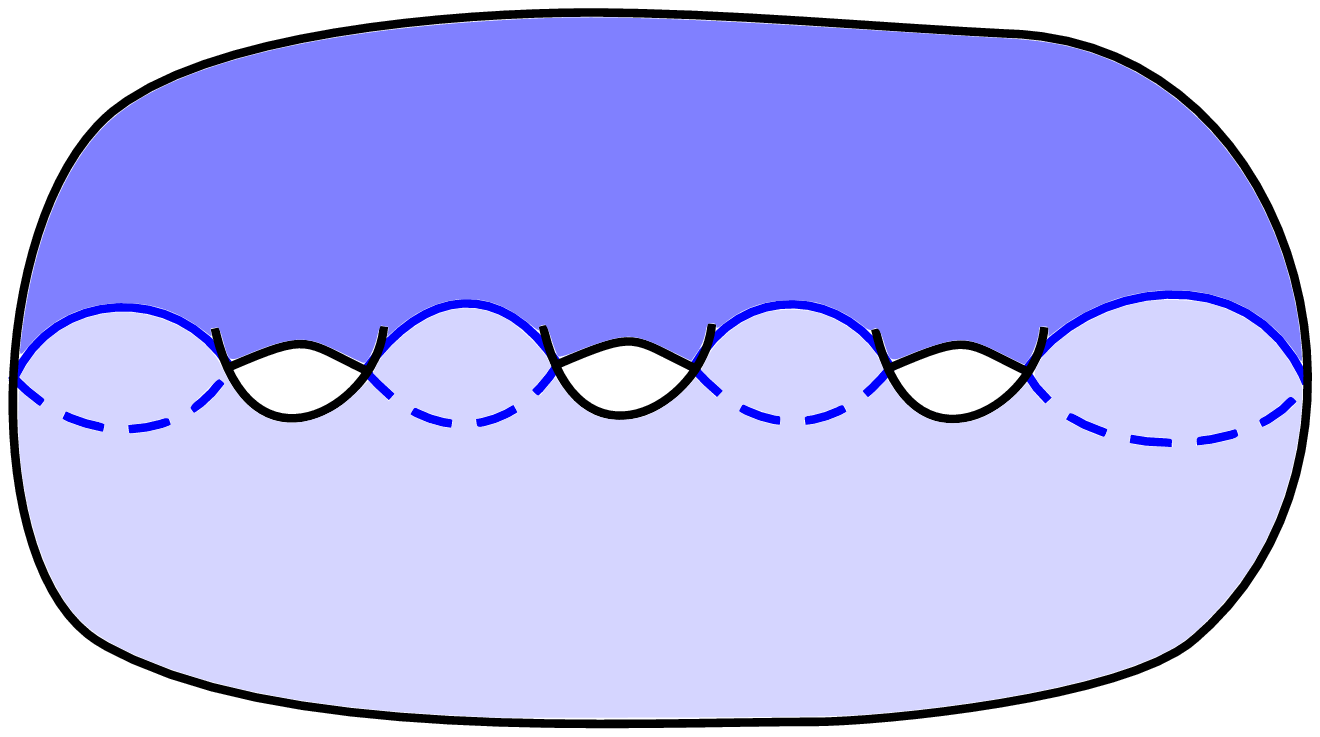}
\subcaption{}
\label{figure: disjointholes1}
\end{subfigure}
\enspace
\begin{subfigure}[b]{0.23\textwidth}
\centering
\includegraphics[width=\textwidth]{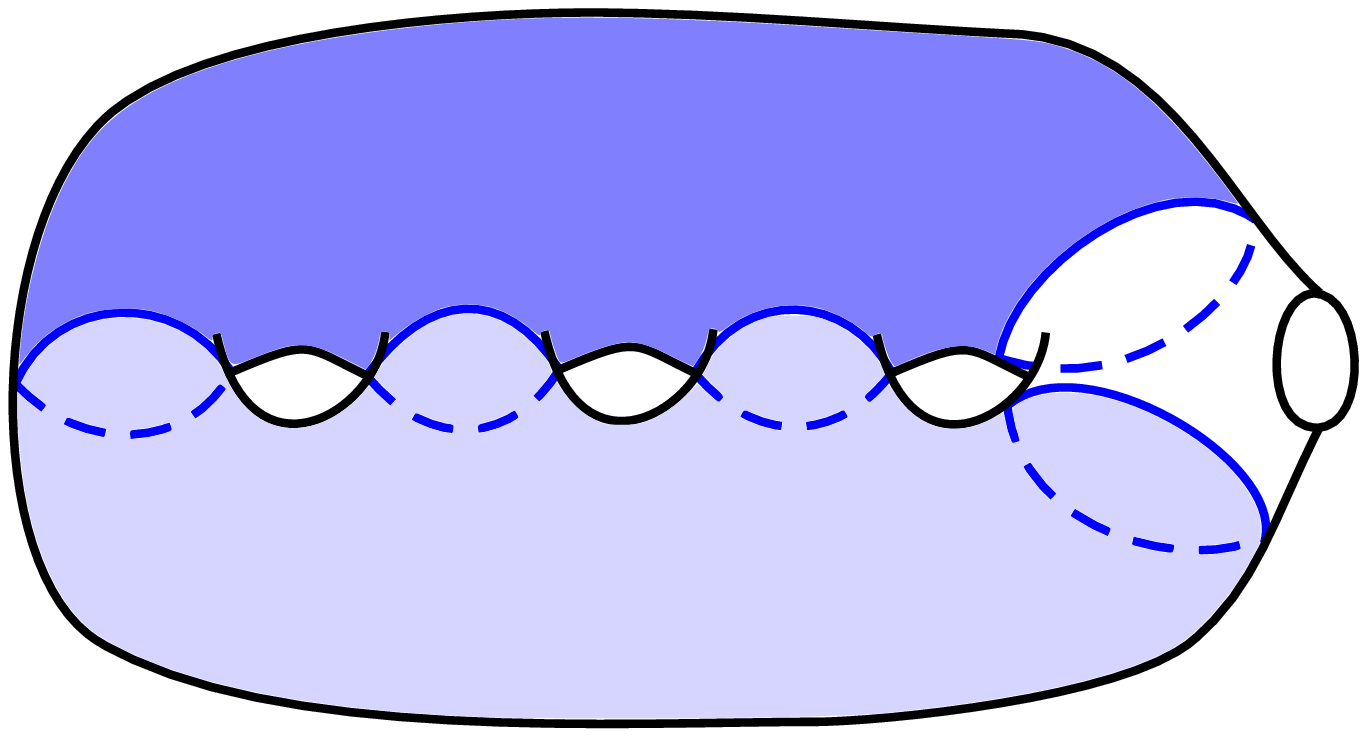}
\subcaption{}
\label{figure: disjointholes2}
\end{subfigure}
\enspace
\begin{subfigure}[b]{0.23\textwidth}
\centering
\includegraphics[width=\textwidth]{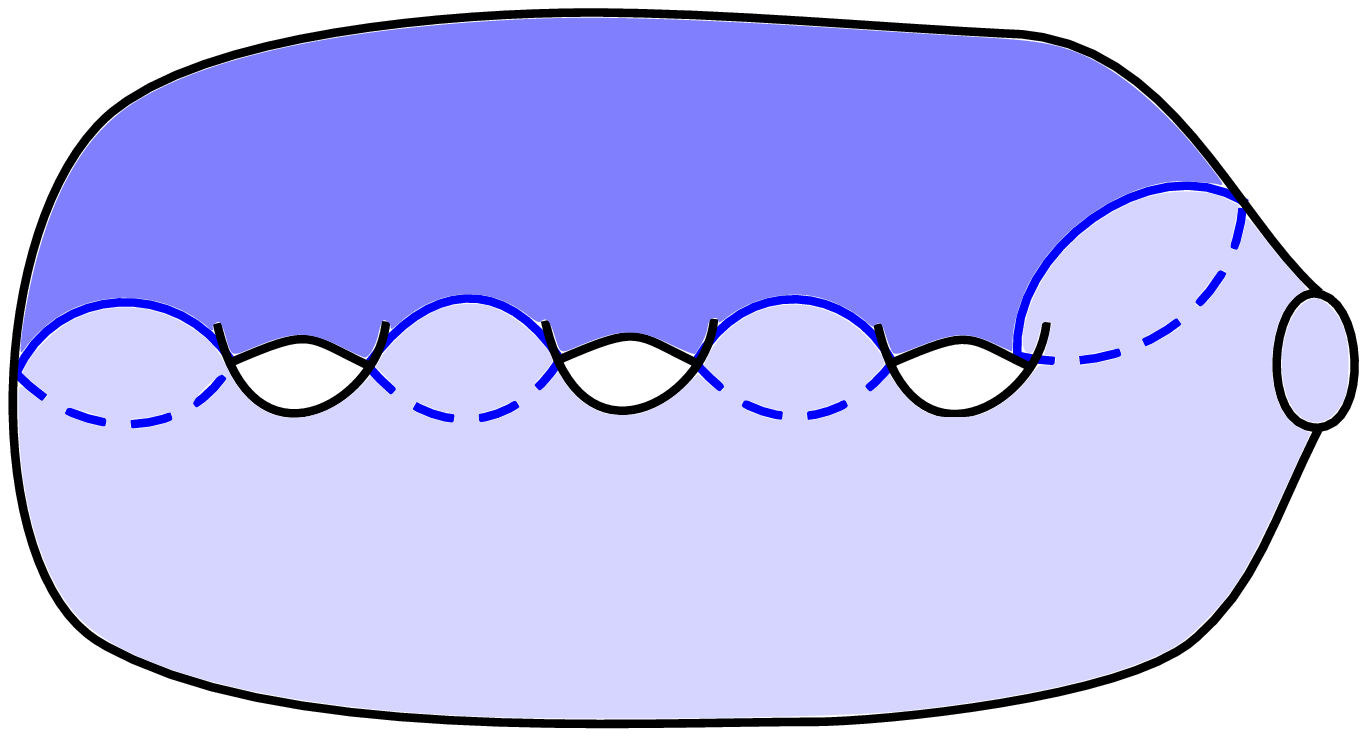}
\subcaption{}
\label{figure: disjointholes2a}
\end{subfigure}
\enspace
\begin{subfigure}[b]{0.22\textwidth}
\centering
\includegraphics[width=\textwidth]{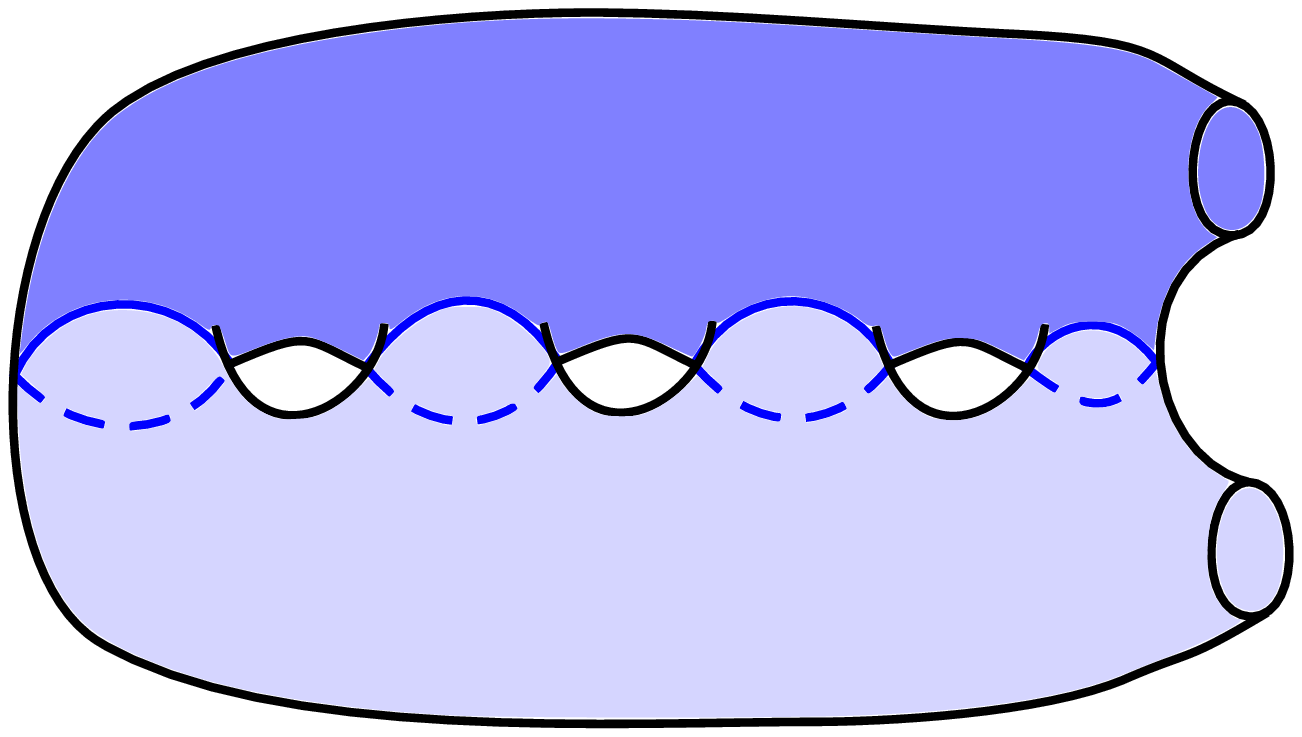}
\subcaption{}
\label{figure: disjointholes3}
\end{subfigure}
\caption{The possibilities for pairs of disjoint witnesses for $\sep(S)$, up to $\mcg(S)$, for $S_3$, $S_{3, 1}$ and $S_{3,2}$.}
\label{figure: disjointholes}
\end{figure}

\begin{example}
The non-separating curve graph, $\text{Nonsep}(S)$, is the full subgraph of $\C(S)$ spanned by non\hyp{}separating curves, and is also a \mygraph{} whenever it is connected (that is, when the genus of $S$ is at least 2).
We can also obtain a connected graph, which is a \mygraph{}, in the genus~1 case by allowing curves to intersect once.
The set of witnesses for $\text{Nonsep}(S)$ is the set containing each subsurface $X$ which has the same genus as $S$, in other words, so that every component of $S \sminus X$ has genus 0 and meets $X$ along a single separating curve.
In particular, when $S$ has at most one boundary component, the only witness is $S$ itself, and $\text{Nonsep}(S)$ is quasi\hyp{}isometric to $\C(S)$ (this observation predates the terminology of hierarchically hyperbolic spaces; see \cite[Exercise~2.39]{sch}).
This hierarchically hyperbolic structure on $\text{Nonsep}(S)$ has been constructed independently by Alexander Rasmussen.
Note that it was already known that $\text{Nonsep}(S)$ is hyperbolic~\cite{hamnonsep,rasmussen}, so that it also has a trivial hierarchically hyperbolic structure with respect to the identity map to itself.
\end{example}

\begin{example}
The cut system graph defined in~\cite{hatchthu} is a \mygraph{}.
The set of witnesses is the set of subsurfaces with positive genus.
The conclusion that the cut system graph is hierarchically hyperbolic with respect to subsurface projections to witnesses appears to be a new observation, though it is closely related to results of Ma~\cite{mahierarchy}.
\end{example}

\begin{example}
The Torelli geometry defined in \cite{farbivanov} can be considered as a \mygraph{} by taking the 1\hyp{}skeleton with the combinatorial metric and forgetting the extra markings which distinguish vertices of different topological types.
\end{example}

\begin{example}
In Appendix~\ref{appendix: arc graph hhs}, we define a quasi\hyp{}isometry from the arc graph to a certain \mygraph{}, allowing us to deduce a hierarchically hyperbolic structure on the arc graph with respect to subsurface projections to witnesses.
\end{example}

\section{Hierarchical hyperbolicity of an associated graph} \label{section: ksep}

In this section, we associate a graph $\ksep_\G(S)$ to each \mygraph{} $\G(S)$, and prove that the graph $\ksep_\G(S)$ is hierarchically hyperbolic.
We shall show in Section~\ref{section: sep} that the graphs $\G(S)$ and $\ksep_\G(S)$ are quasi\hyp{}isometric, and use this to deduce Theorem~\ref{theorem: g hhs}.

\subsection{Definition of $\ksep_\G(S)$} \label{subsection: ksep def}

Let $S$ be a surface and $\G(S)$ a \mygraph{}.
We denote by $\X$ the set of witnesses for $\G(S)$.
Note that when we remove a multicurve $a$ from~$S$, we will really want to remove a regular open neighbourhood in order to obtain compact subsurfaces.
However, we shall abuse notation and simply write $S \sminus a$.
Similarly, for $X$ a subsurface of~$S$, we will write $S \sminus X$ when we really mean $\overline{S \sminus X}$.

\begin{definition} \label{definition: ksep}
The graph $\ksep_\G(S)$ has:
\begin{itemize}
\item a vertex for each multicurve $a$ in $S$ such that every component of $S \sminus a$ is not in $\X$,
\item an edge between vertices $a$ and $b$ if one of the following holds:
\begin{enumerate}
\item $b$ is obtained either by adding a single curve to $a$ or by removing a single curve from $a$,
\item $b$ is obtained from $a$ by a \emph{flip move} as defined below.
\end{enumerate}
\end{itemize}
\end{definition}

\begin{definition} \label{definition: flip move}
Let $a$ be a vertex of $\ksep_\G(S)$.
A \emph{flip move} from $a$ to another multicurve $b$ is defined as follows.
\begin{enumerate}
\item Choose a curve $\alpha$ of $a$.
\item Let $X_\alpha$ be the component of $S \sminus (a \sminus \alpha)$ containing $\alpha$.
\item Choose a curve $\beta$ in $X_\alpha$ such that $\beta$ is adjacent to $\alpha$ in $\C(X_\alpha)$.
\item Let $b = (a \sminus \alpha) \cup \beta$.
\end{enumerate}
\end{definition}

In the case that performing a flip move or removing a curve from $a$ yields a multicurve that is not a vertex of $\ksep_\G(S)$, then this move will not correspond to an edge of~$\ksep_\G(S)$.
Adding a curve to a vertex of $\ksep_\G(S)$ will always give another vertex, since any subsurface containing a witness for $\G(S)$ is itself a witness for~$\G(S)$.

Observe that adjacent vertices of $\ksep_\G(S)$ will intersect at most twice.
Notice also that every vertex of $\G(S)$ is also a vertex of $\ksep_\G(S)$.
Moreover, every pants decomposition of $S$ is a vertex of $\ksep_\G(S)$, since every witness for $\G(S)$ has positive complexity.

\begin{claim} \label{claim: ksep connected}
The graph $\ksep_\G(S)$ is connected.
\end{claim}

Each vertex of $\ksep_\G(S)$ is connected to a pants decomposition by adding curves one by one.
Moreover, a pants move can be realised as a flip move in $\ksep_\G(S)$.
Since the pants graph is connected \cite{hatcherpants}, this implies that $\ksep_\G(S)$ is connected.

\qedclaim

As usual, from now on we shall treat $\ksep_\G(S)$ as a discrete set of vertices equipped with the combinatorial metric induced from the graph.

\begin{claim} \label{claim: ksep sep holes}
Let $\mathfrak{Z}$ be the set of witnesses for $\ksep_\G(S)$.
Then $\mathfrak{Z}=\X$.
\end{claim}

Firstly, $\mathfrak{Z}$ is contained in $\X$ since each vertex of $\G(S)$ is a vertex of $\ksep_\G(S)$.
Suppose $X$ is in $\X$ and $a$ is a vertex of $\ksep_\G(S)$.
If $a$ does not cut $X$ then $X$ is contained in a single component of $S \sminus a$.
But then this component of $S \sminus a$ is in~$\X$, which contradicts that $a$ is a vertex of $\ksep_\G(S)$.

\qedclaim

Note that this means in particular that if $\G(S)$ and $\G'(S)$ are two \mygraph{}s with the same set of witnesses then the graphs $\ksep_\G(S)$ and $\ksep_{\G'}(S)$ are the same.
The graph $\ksep_\G(S)$ is in a sense the ``biggest'' graph with this set of witnesses.

In Sections \ref{subsection: 1 to 8} and \ref{subsection: axiom 9}, we shall prove the following theorem.

\begin{theorem} \label{theorem: ksep hhs}
Let $S$ be a surface and $\G(S)$ a \mygraph{}.
Let $\X$ be the set of witnesses for $\G(S)$.
Then the graph $\ksep_\G(S)$ is a hierarchically hyperbolic space with respect to subsurface projections to the curve graphs of subsurfaces in~$\X$.
\end{theorem}

\subsection{Verification of Axioms 1--8} \label{subsection: 1 to 8}

As above, let $\X$ be the set of witnesses for $\ksep_\G(S)$ (or equivalently for $\G(S)$).
We will verify that $\ksep_\G(S)$ satisfies the axioms for hierarchical hyperbolicity (see Section \ref{subsection: hhs}) for $\mathfrak{S}=\X$.
For each~$X \in \X$, the $\delta$\hyp{}hyperbolic space $\C(X)$ is the curve graph of~$X$.
The constant $\delta$ need not depend on the surface~$S$, since curve graphs are uniformly hyperbolic~\cite{aouhyp,bowhyp,crs,hpw}.
Most of the axioms follow easily from known results on subsurface projections.
The only significant new work needed is the verification of Axiom~9.
We reserve this for a separate section, and verify Axioms 1 to~8 below.

\newcase

\textbf{1. Projections.}
Let $\pi_X\colon \ksep_\G(S) \rightarrow 2^{\C(X)}$ be the usual subsurface projection.
The image of a vertex is never empty since every vertex of $\ksep_\G(S)$ intersects each $X$ in~$\X$.
Let $a$ and $b$ be at distance 1 in~$\ksep_\G(S)$.
First suppose $b$ is obtained from $a$ by adding or removing a curve or by a flip move in a subsurface of complexity at least~2.
Then $a \cup b$ is a multicurve, so its projection to any $\C(X)$ for $X \in \X$ has diameter at most 2 by \cite[Lemma~2.3]{mm2}. 
Suppose $a$ and $b$ are connected by a flip move in a subsurface $X_\alpha$ such that $\xi(X_\alpha)=1$.
If $X = X_\alpha$, then the projection of $a \cup b$ to $\C(X)$ is two adjacent curves and has diameter 1.
Suppose $X \ne X_\alpha$.
Any subsurface of $X_\alpha$ has non\hyp{}positive complexity so cannot be in~$X$.
Hence some curve of $\partial_S X_\alpha$ intersects $X$.
This curve is disjoint from every other curve of $a \cup b$ so the diameter of the projection is at most~4.
Hence, the projection $\pi_X$ is 4\hyp{}Lipschitz.

In order to prove that, for some $K$, the image of each $\pi_X$ is $K$\hyp{}quasiconvex in $\C(X)$, note that this will in particular be true if $N_{\C(X)}(\pi_X(\ksep_\G(S)), K)=\C(X)$.
Now, any curve in $S$ appears in some multicurve which is a vertex of $\ksep_\G(S)$, since every pants decomposition is a vertex.
In particular, every curve in $X$ appears in some vertex of $\ksep_\G(S)$, and hence in the image of this vertex under $\pi_X$.
Hence, the map $\pi_X$ is in fact surjective.

\newcase

\textbf{2. Nesting.}
The partial order on $\X$ is inclusion of subsurfaces, with $X \nest Y$ if $X$ is contained in $Y$.
The unique $\nest$\hyp{}maximal element is $S$.
If $X \nestneq Y$, then we can take $\pi_Y(X) = \partial_Y X \subset \C(Y)$, that is, all boundary curves of $X$ which are non\hyp{}peripheral in $Y$.
This has diameter at most 1 in $\C(Y)$ as the curves are pairwise disjoint.
The projection $\pi^Y_X\colon \C(Y) \rightarrow 2^{\C(X)}$ is the subsurface projection from $\C(Y)$ to $2^{\C(X)}$.

\newcase

\textbf{3. Orthogonality.}
The orthogonality relation $\perp$ on $\X$ is disjointness of subsurfaces.
If $Z$ is disjoint from $Y$ then it is disjoint from any subsurface of $Y$.
Suppose $X \in \X$ and $Y \nest X$.
Then either no other subsurface of $X$ disjoint from $Y$ is in $\X$, or the complement $Z=X \sminus Y$ is in $\X$ and any $U \nest X$ which is disjoint from $Y$ is nested in $Z$.
Finally, if $X$ and $Y$ are disjoint then neither is nested in the other.

\newcase

\textbf{4. Transversality and consistency.}
Two subsurfaces $X$ and $Y$ in $\X$ are transverse, $X \pitch Y$, if they are neither disjoint nor nested.
If $X \pitch Y$, let $\pi_X(Y)$ be the subsurface projection of $\partial_S Y \subset \C(S)$ to $\C(X)$, and similarly for $\pi_Y(X)$.
These each have diameter at most 2 by \cite[Lemma~2.3]{mm2}.
By Behrstock's lemma \cite[Theorem~4.3]{beh}, for each $S$ there exists $\kappa$ such that for any $X \pitch Y$ and any multicurve $a$ projecting to both (and hence any vertex $a$ of $\ksep_\G(S)$),
\[ \min \lbrace d_X(\pi_X(a), \pi_X(Y)), d_Y(\pi_Y(a), \pi_Y(X)) \rbrace \le \kappa. \]
For a more elementary proof due to Leininger, with a uniform value of $\kappa$, see \cite[Lemma~2.13]{man}.
Given $X \nest Y$, and $a$ in $\ksep_\G(S)$ consider
\[ \min \lbrace d_Y(\pi_Y(a), \pi_Y(X)), \text{diam}_{\C(X)}(\pi_X(a) \cup \pi^Y_X(\pi_Y(a))) \rbrace. \]
The second term compares projecting $a$ directly to $\C(X)$ from $\ksep_\G(S)$ and projecting $a$ first to $\C(Y)$ and then to~$\C(X)$.
This gives coarsely the same result, so that this term is bounded.
Finally, if $X \nest Y$, then the union of their boundary components is a multicurve in~$\C(S)$, so for any $Z \in \X$ such that each of $X$ and $Y$ is either transverse to $Z$ or strictly nested in~$Z$, we have $d_Z(\pi_Z(X), \pi_Z(Y)) \le 2$.

\newcase

\textbf{5. Finite complexity.}
The length of a chain of nested subsurfaces in $\X$ is bounded above by $\xi(S)$.

\newcase

\textbf{6. Large links.}
Let $X \in \X$ and $a, b \in \ksep(S)$, with $R = d_X(a, b)+1$.
Assume for now that~$\xi(X) \ge 2$.
Let $\gamma_1, \gamma_2, \dots, \gamma_{R-1}, \gamma_R$ be a geodesic in~$\C(X)$, where $\gamma_1 \in \pi_X(a)$ and $\gamma_R \in \pi_X(b)$.
For each $1 \le i \le R$, let $Y_i$ be the component of $X \sminus \gamma_i$ containing the adjacent curves of the geodesic.
Note that $Y_i$ is not necessarily in~$\X$.

Suppose $Y \in \X$ satisfies $Y \nestneq X$ and $d_Y(a, b) > M$, where $M$ is the constant of \cite[Theorem~3.1]{mm2} (Bounded Geodesic Image; see also Axiom~7 below for more detail).
The Bounded Geodesic Image Theorem implies that, in this case, some $\gamma_i$ does not intersect~$Y$.
Hence $Y$ is contained in a single component of $S \sminus \gamma_i$.
Suppose that this component is not~$Y_i$.
Then the adjacent curves to $\gamma_i$ in the geodesic also do not cut~$Y$, by the definition of~$Y_i$.
Since $S \sminus Y_i$ is contained in $Y_{i-1}$ or $Y_{i+1}$, so too is~$Y$.
Hence, $Y$ is contained in some~$Y_i$.
We also need to check that this $Y_i$ is in~$\X$.
This follows from the fact that $Y$ is in~$\X$, and hence so is any subsurface containing~$Y$.

We include only those $Y_i$ which are in $\X$ in the list.
If there are no subsurfaces of $\X$ properly nested in~$X$, and, in particular, if~$\xi(X)=1$, then trivially $d_Y(a, b) \le M$ for every $Y \in \X$ with $Y \nestneq X$.
Finally, for each~$i$, we have $d_X(\pi_X(a), \pi_X(Y_i))=d_X(\pi_X(a), \pi_X(\gamma_i)) \le R$.

\newcase

\textbf{7. Bounded geodesic image.}
By \cite[Theorem~3.1]{mm2}, there exists $M$ so that for all $Y \nestneq X$, and any geodesic $g$ in $\C(X)$, either $\diam_{\C(Y)}(g) \le M$ or some vertex $\gamma$ of $g$ does not intersect $Y$.
If $\gamma$ is disjoint from $Y$, then it is adjacent in $\C(X)$ to $\pi_X(Y)=\partial_X Y$.
Hence, if $\diam_{\C(Y)}(g) > M$, then $g \cap N_{\C(X)}(\pi_X(Y), 1) \ne \varnothing$, and so the conditions of this axiom are satisfied for $E=M$.
For a proof that the constant $M$ does not depend on the surface $S$, see \cite{webb}.

\newcase

\textbf{8. Partial realisation.}
Let $\lbrace X_j \rbrace$ be a set of pairwise disjoint elements of $\X$, and let $\gamma_j$ be a curve in $X_j$ for each $j$.
We need to find a vertex $a$ of $\ksep_\G(S)$ with projections at bounded distance from $\gamma_j$ in each $\C(X_j)$, and at bounded distance from $\partial_S X_j$ for other subsurfaces in $\X$.
First define a multicurve $a' = \bigcup_j \gamma_j \cup \bigcup_j \partial_S X_j$.
Now add curves to complete $a'$ to a pants decomposition $a$.
As previously observed, this must be a vertex of $\ksep_\G(S)$ since every subsurface in $\X$ has positive complexity.
For each~$j$, the projection of $a$ to $X_j$ is a multicurve containing $X_j$, so $d_{X_j}(\pi_{X_j}(a), \gamma_j) \le 1$.
Furthermore, suppose that $X$ is a subsurface of $S$ containing $X_1$.
Since $a$ contains $\partial_X X_1$, $d_X(\pi_X(a), \pi_X(X_j)) = \diam_{\C(X)}(\pi_X(a)) \le 2$.
Similarly, if $Y \in\X$ is transverse to $X_j$, then $d_Y(\pi_Y(a), \pi_Y(X_j)) \le 2$.

\qedclaim

We remark that all of the above constants, apart from the complexity, may be taken to be independent of the surface $S$.
Our proof below that Axiom~9 holds gives constants which do depend on the surface $S$ and are probably far from optimal.
It would be interesting to consider how far they can be improved.
The quasi\hyp{}isometry constants in Section \ref{section: sep} also \emph{a priori} depend on the surface.

\subsection{Verification of Axiom 9} \label{subsection: axiom 9}

The most significant part of the proof of Theorem \ref{theorem: ksep hhs} is the verification of the final axiom.
For brevity of notation, we will now suppress the projection maps when considering distances and diameters for subsurface projections.

\begin{proposition} \label{proposition: ksep bound}
Let $S$ be a surface and $\G(S)$ a \mygraph{}.
For every $K$, there exists~$K'$, depending only on $K$ and the graph $\G(S)$, such that if $a$~and $b$ are two vertices of $\ksep_\G(S)$, and if $d_X(a, b) \le K$ for every subsurface $X$ in~$\X$, then $d_{\ksep_\G(S)}(a, b) \le K'$.
\end{proposition}

In order to prove this, we make use of a combinatorial construction based on that described in \cite[Section~10]{bowelt}.
This will give us a way of representing a sequence of multicurves in $S$.
We shall construct this sequence inductively so that eventually it will be a path in $\ksep_\G(S)$.
We remark that this method is also related to the hierarchy machinery of \cite{mm2}.

We shall consider the product $S \times I$, for a non\hyp{}trivial closed interval $I$.
We consider $S$ to be the \emph{horizontal} direction and $I$ to be the \emph{vertical} direction.
We have a vertical projection $S \times I \rightarrow S$ and a horizontal projection $S \times I \rightarrow I$.
When we denote a subset of $S \times I$ by $A_1 \times A_2$, $A_1$ will be a subset of the horizontal factor, $S$, and $A_2$ of the vertical factor, $I$.
To ensure that curves in $S$ are pairwise in minimal position, we will fix a hyperbolic structure on $S$ with totally geodesic boundary and take the geodesic representative of each isotopy class of curves.

\begin{definition}
A \emph{vertical annulus} in $S \times I$ is a product $\gamma \times I_\gamma$, where $\gamma$ is a curve in $S$ and $I_\gamma$ is a non\hyp{}trivial closed subinterval of $I$.
The curve $\gamma$ is the \emph{base curve} of the annulus.
\end{definition}

\begin{definition}
An \emph{annulus system} $W$ in $S \times I$ is a finite collection of disjoint vertical annuli.
An annulus system $W$ is \emph{generic} if whenever $\gamma_1 \times I_1$ and $\gamma_2 \times I_2$ are two distinct annuli in $W$, we have $\partial I_1 \cap \partial I_2 \subseteq \partial I$.
\end{definition}

We denote $S \times \lbrace t \rbrace$ by $S_t$ and $W \cap S_t$ by $W_t$.
Each $W_t$ is a (possibly empty) multicurve, and there is a discrete set of points in $I$ where the multicurve $W_t$ changes.
Hence the annulus system is a way of recording a sequence of multicurves in $S$.

\begin{definition}
Let $\xi(S) \ge 2$.
A \emph{tight geodesic} in $\C(S)$ between curves $\gamma$ and $\gamma'$ is a sequence $\gamma=v_0, v_1, \dots, v_{n-1}, v_n=\gamma'$, where:
\begin{itemize}
\item each $v_i$ is a multicurve in $S$,
\item for any $i \ne j$ and any curves $\gamma_i \in v_i$, $\gamma_j \in v_j$,  $d_S(\gamma_i, \gamma_j)= \lvert i-j \rvert$,
\item for each $1 \le i \le n-1$, $v_i$ is the boundary multicurve of the subsurface spanned by $v_{i-1}$ and $v_{i+1}$ (excluding any components of $\partial S$).
\end{itemize}
If $\xi(S)=1$, then a tight geodesic is an ordinary geodesic in $\C(S)$.
\end{definition}

Note that this is called a \emph{tight sequence} in \cite[Definition~4.1]{mm2}.
The tight geodesics of~\cite{mm2} are also equipped with initial and terminal markings.
A tight geodesic can be realised as an annulus system as follows.

\begin{definition}
A \emph{tight ladder} in $S \times I$ is a generic annulus system $W$ so that:
\begin{itemize}
\item there exists a tight geodesic $v_0, v_1, \dots, v_{n-1}, v_n$ in $\C(S)$ so that the curves appearing in the tight geodesic correspond exactly to the base curves of the annuli in $W$,
\item for two annuli $\gamma \times I_\gamma$ and $\delta \times I_\delta$ in $W$, the intervals $I_\gamma$ and $I_{\delta}$ overlap if and only if $\gamma$ and $\delta$ are disjoint,
\item there exist $t_0 < t_1 < \dots < t_{n-1} < t_n$ in $I$ such that for each $i$ the multicurve $W_{t_i}=v_i$.
\end{itemize}
\end{definition}

In the case where $\xi(S) \ge 2$, this corresponds to moving from $v_i$ to $v_{i+1}$ by adding in the curves of $v_{i+1}$ one at a time then removing the curves of $v_i$ one at a time (Figure \ref{figure: tightladder1}).
In the case where $\xi(S) = 1$, this corresponds to moving from $v_i$ to $v_{i+1}$ by removing the curve $v_i$ then adding in the curve $v_{i+1}$ after a vertical interval with no annuli (Figure \ref{figure: tightladder2}).

\begin{figure}[h!]
\begin{subfigure}[b]{0.48\textwidth}
\centering
\includegraphics[width=.6\textwidth]{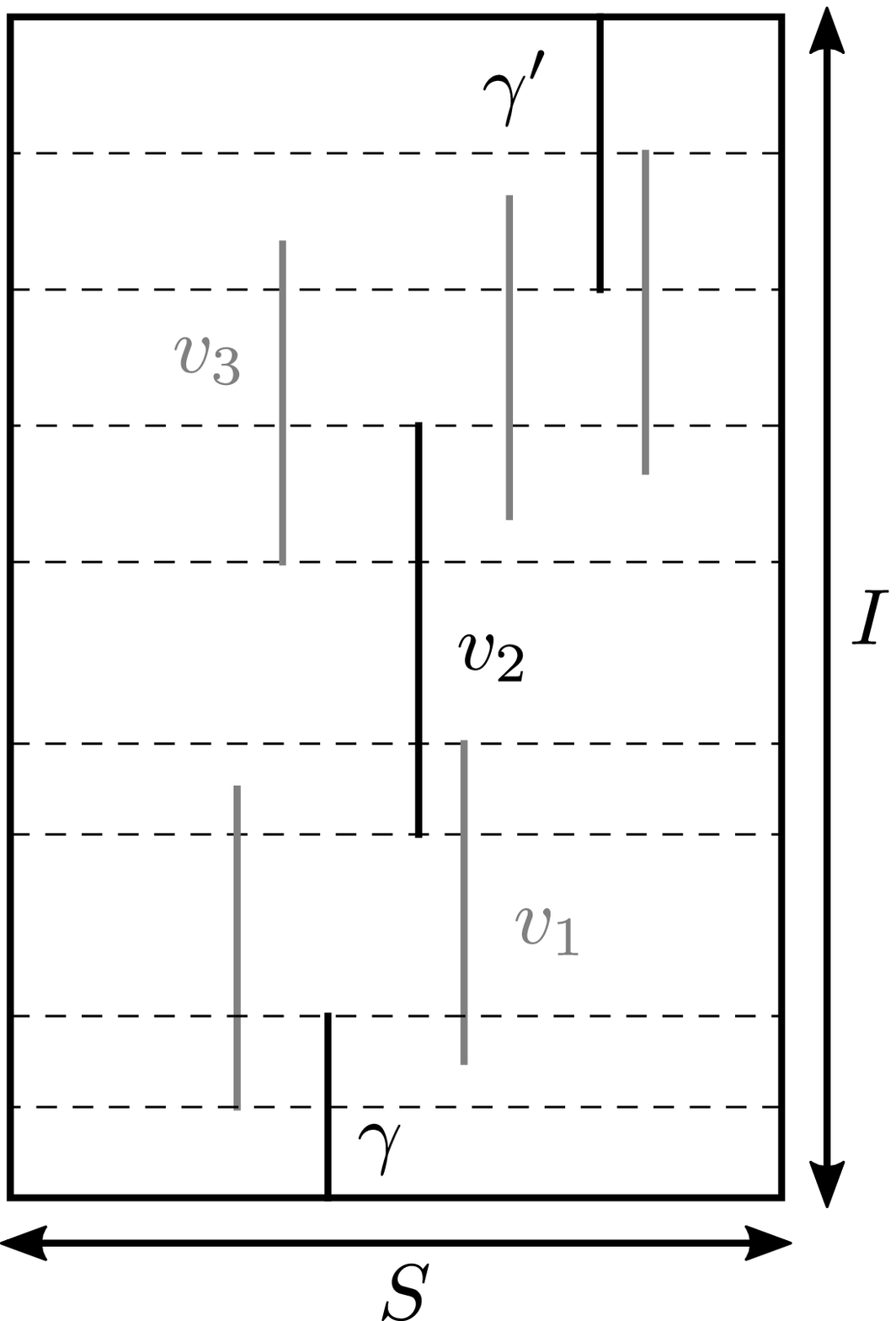}
\subcaption{Complexity $\xi(S) \ge 2$.}
\label{figure: tightladder1}
\end{subfigure}
\begin{subfigure}[b]{0.48\textwidth}
\centering
\includegraphics[width=.6\textwidth]{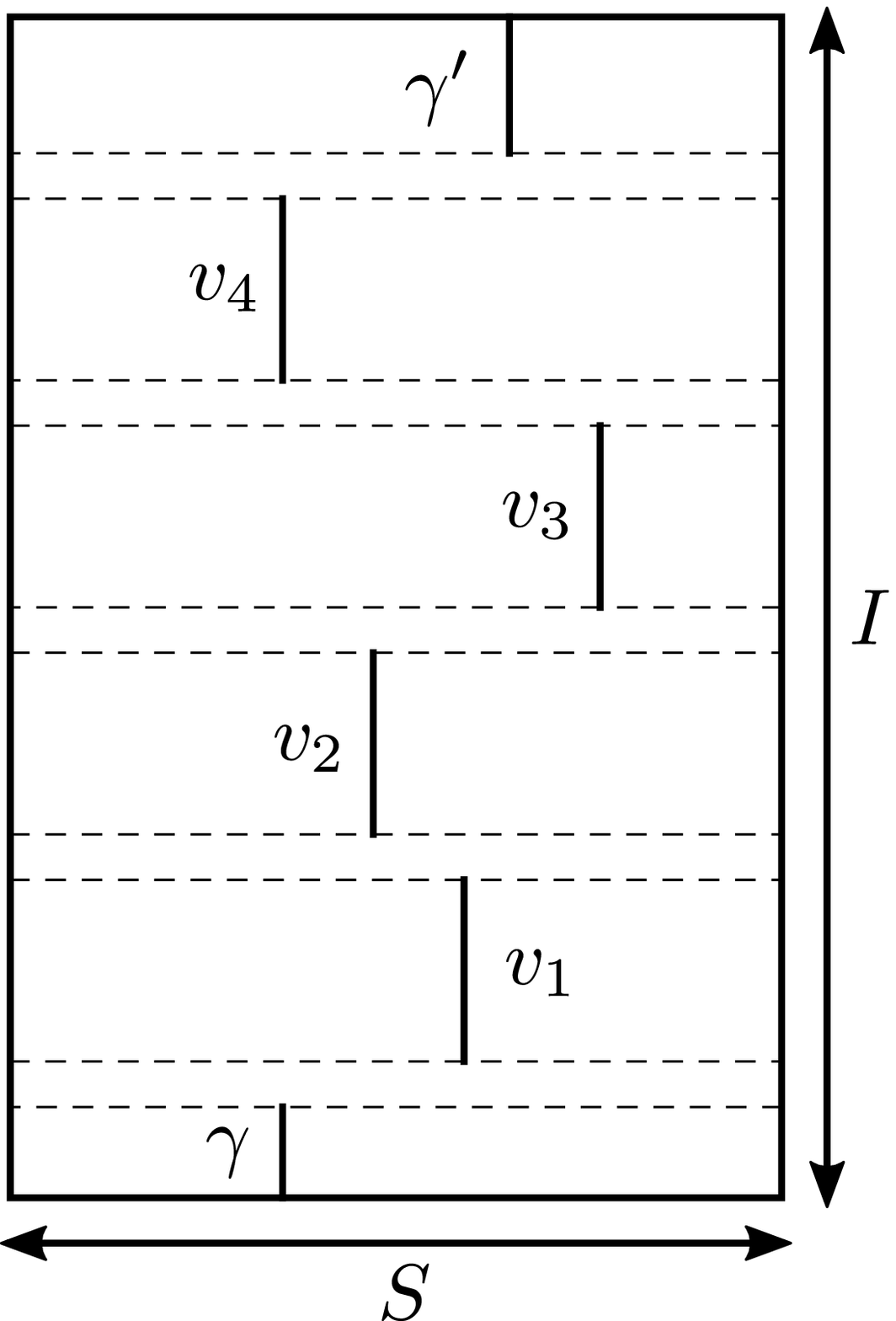}
\subcaption{Complexity $\xi(S)=1$.}
\label{figure: tightladder2}
\end{subfigure}
\caption{Illustrations of tight ladders in $S \times I$.}
\end{figure}

\begin{definition}
Let $t \in I$, and let $X$ be a component of $S_t \sminus W_t$.
Let $J \subseteq I$ be the maximal interval containing $t$ such that $X$ is a component of $S_s \sminus W_s$ for every $s \in J$.
The product $X \times \bar{J}$ is a \emph{brick} of $W$.
The surface $X$ is the \emph{base surface} of the brick.
\end{definition}

We remark that this differs slightly from the definition of ``brick'' in~\cite{bowelt}.
Note that the interiors of any two distinct bricks are disjoint, and that we may decompose $S \times I$ as a union of regular neighbourhoods of all bricks of~$W$ (recall that when we remove a multicurve $a$ from $S$, we also remove a regular open neighbourhood of $a$).
In order to obtain a path in $\ksep_\G(S)$, we want to decompose $S \times I$ into bricks whose base surfaces are not in~$\X$.

\begin{definition}
A brick $X \times [s, t]$ is \emph{small} if one of the following holds.
\begin{enumerate}
\item[(Type 1)] The base surface $X$ is not in~$\X$.
\item[(Type 2)] The base surface $X$ has complexity 1 and is in~$\X$.
Moreover, $W_s$ and $W_t$ each intersect $X$ in an essential non\hyp{}peripheral curve, and the two curves are adjacent in~$\C(X)$.
\end{enumerate}
\end{definition}

Notice that a generic annulus system~$W$ where every brick is small realises a path in~$\ksep(S)$, as follows.
If a cross\hyp{}section $S_t$ of $S \times I$ intersects only Type~1 small bricks, then the multicurve $W_t$ is a vertex of $\ksep_\G(S)$.
First, for simplicity, let us assume that all of the bricks in~$W$ are Type~1 small bricks.
The multicurves $W_t$ for $t \in I$ change precisely at the points in the interior of~$I$ which are the endpoints of horizontal projections of annuli in~$W$.
Let $P$ denote this set of points.
Let $I_0, \dots, I_n$ be the components of~$I \sminus P$ in the order in which they appear in~$I$, and for each $0 \le j \le n$ pick any $t_j$ from~$I_j$.
Let $a_j$ be the multicurve~$W_{t_j}$.
The sequence $a_0, \dots, a_n$ is a path in~$\ksep_\G(S)$ given by successive moves of adding and removing curves.

In the case where $\X$ contains complexity 1 subsurfaces, we place an additional restriction on a generic annulus system, requiring that whenever we have a Type~2 small brick, the endpoints of its horizontal projection to~$I$ are consecutive points of~$P$.
This can be achieved by appropriate isotopies.
Again, let $W$ be a generic annulus system where every brick is small.
Construct the sequence of curves~$a_j$ as above and suppose that, for some~$j$, $S \sminus a_j$ has a component~$X$ which is a complexity~1 subsurface in~$\X$ (and hence $a_j$ is not a vertex of~$\ksep_\G(S)$).
Then by the restriction on the endpoints of the horizontal projection of a Type~2 small brick, $X$ is not a component of $S \sminus a_{j-1}$ or~$S \sminus a_{j+1}$, and neither is any other complexity~1 subsurface which is in~$\X$.
Then $a_{j-1}$ and~$a_{j+1}$ are vertices of~$\ksep_\G(S)$, and moreover, by the definition of a Type~2 small brick, they are adjacent in this graph.
Hence we obtain a path in~$\ksep_\G(S)$ as for the previous case, except that we must remove any multicurves in the sequence $a_0, \dots, a_n$ which are not vertices of~$\ksep_\G(S)$.

\begin{sloppypar}
\begin{definition}
The $\ksep$\hyp\emph{complexity} of an annulus system $W$ is $(N_{\xi(S)}, N_{\xi(S)-1}, \dots, N_1)$, where, for each $i$, $N_i$ is the total number of non\hyp{}small bricks of $W$ whose base surface is a subsurface in $\X$ of complexity $i$.
We give this the lexicographical ordering.
\end{definition}
\end{sloppypar}

Since there are no subsurfaces in $\X$ of complexity less than 1, the $\ksep$\hyp{}complexity is $(0, 0, \dots, 0)$ precisely when every brick is small.

\qedclaim

We now begin the proof of Proposition \ref{proposition: ksep bound}.
Let $I=[0,1]$.
We shall construct a generic annulus system in $S \times I$, with $\ksep$\hyp{}complexity $(0, 0, \dots, 0)$, which realises a path in $\ksep_\G(S)$ from $a$ to $b$, and show that the length of this path is bounded in terms of $K$.

We construct the annulus system inductively.
We start by choosing distinct points $t_{\alpha} \in (0, \frac{1}{2})$ for each curve $\alpha$ of $a$ and $t_{\beta} \in (\frac{1}{2}, 1)$ for each curve $\beta$ of $b$ and defining an annulus system $\W0=\bigcup_{\alpha}(\alpha \times [0, t_\alpha]) \cup \bigcup_{\beta}(\beta \times [t_\beta, 1])$.

We will describe below the procedure for constructing a new annulus system $\W{k+1}$ from $\W{k}$, where the first annulus system $\W0$ is as defined above.
We shall do this in such a way that each annulus system interpolates between $a$ and~$b$ (in fact, $\W{k+1}$ contains $\W{k}$), and such that the $\ksep$-complexity of $\W{k+1}$ is strictly less than that of $\W{k}$.
This process will eventually terminate with an annulus system with $\ksep$\hyp{}complexity $(0, 0, \dots, 0)$.

Suppose we have constructed a generic annulus system~$\W{k}$.
We will describe how to construct the next stage~$\W{k+1}$; see Figure~\ref{figure: induction} for an illustration.
Consider the bricks of~$\W{k}$.
If every brick is small, then the $\ksep$\hyp{}complexity of~$\W{k}$ is $(0, \dots, 0)$ and we are done.
Suppose this is not the case, and choose a brick $Y \times [t_-,t_+]$, where $Y$ is in~$\X$ and has maximal complexity among such bricks.
(Note that \emph{a priori} the same subsurface $Y$ might appear as the base surface of more than one brick.)
Decreasing past~$t_-$ and increasing past~$t_+$, the components of $S_t \sminus \W{k}_t$ change to not include~$Y$.
Since $Y$ has maximal complexity among base surfaces of~$\W{k}$ in~$\X$, it is not a proper subsurface of any component of $S_t \sminus \W{k}_t$ for any~$t \in I$.
Hence, the intersection of $\W{k}_{t_-}$ and of~$\W{k}_{t_+}$ with~$Y$ must be non\hyp{}empty, and, since $\W{k}$ is generic, it is in each case a single curve, which we call $\gamma_-$ and~$\gamma_+$ respectively.
Slightly extend $[t_-,t_+]$ on each side to $J=[t_--\epsilon,t_++\epsilon]$ so that the subset $Y \times J$ now contains vertical annuli corresponding to each of these curves but still intersects no other annuli.
We may consider annulus systems in $Y \times J$ as for $S \times I$.
Add a tight ladder in $Y \times J$, corresponding to a tight geodesic in~$\C(Y)$ from $\gamma_-$ to~$\gamma_+$, arranging that the resulting annulus system in $S \times I$ is generic by slightly moving the endpoints of intervals if necessary.
The annulus system $\W{k+1}$ is the union of~$\W{k}$ and the tight ladder in~$Y \times J$.
Notice that the $\ksep$\hyp{}complexity of $\W{k+1}$ is strictly less than that of~$\W{k}$.

\begin{figure}[h!]
\centering
\includegraphics[width=0.3\textwidth]{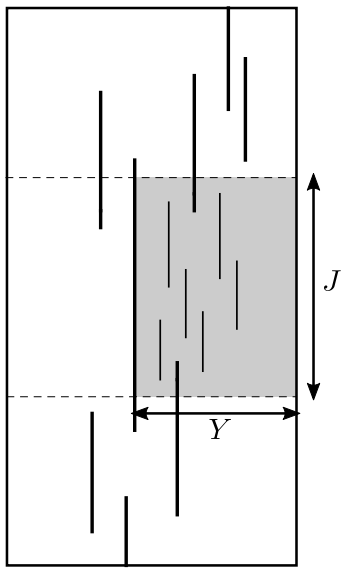}
\caption{Constructing $\W{k+1}$ from $\W{k}$ by adding a tight ladder in a brick $Y \times J$.}
\label{figure: induction}
\end{figure}

At each stage, we add a tight ladder in some brick, $Y \times J$, increasing the length of the sequence of multicurves determined by the annulus system, where these multicurves are not yet necessarily vertices of $\ksep(S)$.
Let us consider the maximal increase in the length of this sequence, in terms of the length of the ladder added.

\begin{claim} \label{claim: ladder increase}
Let $Y \times J$ be a brick in the annulus system~$\W{k}$.
Suppose that we add a tight ladder $v_0, v_1, \dots, v_{n-1}, v_n$ in $Y \times J$ to obtain $\W{k+1}$.
Then the difference in the lengths of the sequences of multicurves determined by $\W{k}$ and $\W{k+1}$ is at most $2n\xi(Y) - 2$.
\end{claim}

Let us consider the set $P_k$ of points in the interior of $I$ which are endpoints of horizontal projections of bricks in $\W{k}$ to~$I$.
Since these precisely correspond to places where the multicurve $\W{k}_t$ changes, the cardinality of $P_k$ is one less than the length of the sequence of multicurves corresponding to~$\W{k}$.
Now, let $Q$ be the set of points in $J$ corresponding to the endpoints of the horizontal projections of bricks to~$J$.
Two of these points, considered as points of~$I$, are already in~$P_k$, and the rest are not.
The set $P_{k+1}$ of endpoints of horizontal projections of bricks in $\W{k+1}$ is exactly $P_k \cup Q$, so the increase in length is $\vert Q \vert -2$.

We now bound the cardinality of~$Q$.
First suppose $\xi(Y) \ge 2$.
The transition from $v_i$ to $v_{i+1}$ gives a point of $Q$ for every curve in $v_i$ and every curve in $v_{i+1}$, so $\vert Q \vert = (\vert v_0 \vert + \vert v_1 \vert) + (\vert v_1 \vert + \vert v_2 \vert) + \dots + (\vert v_{n-1} \vert + \vert v_n \vert) \le n\xi(Y)$.
Now suppose $\xi(Y) =1$.
Then the number of points of $Q$ is $2n=2n\xi(Y)$.
Hence between $\W{k}$ and $\W{k+1}$, when we add a tight ladder of length $n$ in a brick $Y \times J$, we add at most $2n\xi(Y)-2$ to the length of the corresponding sequences of curves.

\qedclaim

The length of the tight ladder we add between $\W{k}$ and $\W{k+1}$ is equal to $d_Y(\gamma_-, \gamma_+)$.
We now show that this quantity is bounded above in terms of $k$ and~$K$.

\begin{claim} \label{claim: bounded projection}
Let $\Gam{k}$ be the set of the base curves of all annuli in $\W{k}$ and $K$ as in the statement of Proposition~\ref{proposition: ksep bound}.
Let $M$ be the constant of the Bounded Geodesic Image Theorem \cite[Theorem~3.1]{mm2}.
Then $\diam_{\C(X)}(\pi_X(\Gam{k})) \le K + 2Mk$ for each~$X \in \X$.
\end{claim}

We prove this by an induction on $k$.
The base case is when $k=0$ and holds since, by hypothesis, $\diam_{\C(X)}(a \cup b) \le K$ for every $X \in \X$.

Now, suppose at stage $k-1$ the projection has diameter at most $K + 2M(k-1)$.
At stage $k$, we add a tight geodesic $g$ in $\C(Y)$ for some $Y \in \X$, where the first and last terms in $g$ are curves which already appear as base curves in $\W{k-1}$.
There are several cases depending on how the subsurface $X$ to which we are projecting intersects $Y$.

\emph{Case 1:} $X$ is disjoint from $Y$.
Then none of the curves added in $Y$ contributes to the projection to $X$ so the diameter is unchanged.

\emph{Case 2:} $X$ intersects $Y$ and is not nested in $Y$.
Then there is a curve $\delta$ in $\partial_S Y$ which intersects $X$ non\hyp{}trivially.
Such a curve is also a base curve in $\Gam{k-1}$.
Every curve added in $Y$ is disjoint from $\delta$.
Hence every curve added either does not intersect $X$ so does not change the projection to $\C(X)$, or projects to a curve at distance at most 2 from $\pi_X(\delta)$.
Hence, the diameter of the projection increases by at most 4.

\emph{Case 3:} $X$ is nested in $Y$.
First suppose that every vertex of the tight geodesic, $g$, we add in $Y$ intersects~$X$.
Then by the Bounded Geodesic Image Theorem, the diameter of $\pi_X(g)$ is at most~$M$.
Moreover, some of the vertices in $\pi_X(g)$ were also contained in $\pi_X(\Gam{k-1})$, so the diameter increases by at most $M$.

Now suppose that not every vertex of $g$ intersects $X$.
Since $g$ is a tight geodesic, the vertices which do not cut $X$ are a sequence of consecutive terms \cite[Lemma~4.10]{mm2}.
Therefore, there are subpaths of the geodesic $g$ where all vertices intersect $X$ - two such if both $v_0$ and $v_n$ intersect $X$, and one if only one of the end vertices intersects~$X$.
If neither of $v_0$ and $v_n$ intersects $X$ then the entire geodesic is disjoint from $X$ and hence does not change the projection to~$X$.
The subpaths of $g$ where every vertex intersects $X$ are necessarily geodesic in $\C(Y)$, and moreover, a vertex from each was already contained in~$\Gam{k-1}$.
Hence, again by the Bounded Geodesic Image Theorem, we have $\diam_{\C(X)}(\Gam{k}) \le \diam_{\C(X)}(\Gam{k-1})+2M \le K+2Mk$.

\qedclaim

In order to find an upper bound on the length of the final path in $\ksep(S)$, we will find upper bounds on the length of the sequence of multicurves at certain stages of the induction.
It will also be useful to bound the $\ksep$\hyp{}complexity in terms of the length of the sequence, and we will obtain such an upper bound from the following claim.

\begin{claim} \label{claim: length to bricks}
Let $\mathcal{W}$ be an annulus system and $n$ the length of the corresponding sequence of multicurves.
Then the number $B$ of bricks in $\mathcal{W}$ is at most $\frac{3}{2}n+\xi(S)$. 
\end{claim}

Consider the set of points $P$ in the interior of $I$ corresponding to the endpoints of horizontal projections of annuli in~$W$.
The cardinality of $P$ is $n-1$.
Each brick which does not meet $S \times \lbrace 0, 1 \rbrace$ projects to an interval in $I$ whose endpoints are exactly two points of $P$.
If a brick does meet $S \times \lbrace 0, 1 \rbrace$ then it corresponds to just one point of~$P$.
There are at most $2\xi(S)$ such end bricks.
On the other hand, each point of $P$ corresponds to a curve in the base surface of a brick.
This curve can be cut along to give either one or two new bricks, and hence meets up to three bricks in total.
Hence we have
\begin{align*}
2B - 2\xi(S) &\le 3(n-1) \\
B &\le \frac{3}{2}n +\xi(S).
\end{align*}

\qedclaim

Now, for each $1 \le i \le \xi(S)$, let $k_i$ be minimal such that $\N{k_i}{j}=0$ for all $i \le j \le \xi(S)$.
That is, the $k_i$-th stage is the stage where the last non\hyp{}small brick of complexity $i$ has been filled in, and we will start to fill in bricks of complexity $i-1$, if witnesses of this complexity exist.
In particular $k_{\xi(S)} \le k_{\xi(S)-1} \le \dots \le k_1$, and the $\ksep$\hyp{}complexity of $\W{k_1}$ is $(0, 0, \dots, 0)$.

\begin{claim} \label{claim: bounded length}
Let $K$ be as in Proposition \ref{proposition: ksep bound} and $M$ the constant of the Bounded Geodesic Image Theorem.
Define $(T_i, L_i)$ inductively by $T_{\xi(S)}=(2K+2)\xi(S)$, $L_{\xi(S)}=1$, $T_{i}=T_{i+1}+4T_{i+1}(K+2ML_{i+1})\xi(S)+8MT_{i+1}^2\xi(S)$ for $1 \le i \le \xi(S)-1$ and $L_i=L_{i+1}+2T_{i+1}$ for $1 \le i \le \xi(S)-1$.
Then for each $1 \le i \le \xi(S)$, $k_i \le L_i$ and the length of the sequence of multicurves determined by $\W{k_i}$ is at most $T_i$.
\end{claim}

We will prove this by a reverse induction on $i$.
The base case is when $i=\xi(S)$.
The initial annulus system $\W{0}$ contains a brick with base surface~$S$, so the first term of the $\ksep$\hyp{}complexity is $\N{0}{\xi(S)}=1$.
The length of the sequence of multicurves defined by $\W{0}$ is at most $\lvert a \rvert + \lvert b \rvert +1 \le 2\xi(S) + 1$.
To get from $\W{0}$ to~$\W{1}$, we add a tight geodesic in~$S$.
The first term of the $\ksep$\hyp{}complexity is now~0, so $k_{\xi(S)}=1=L_{\xi(S)}$.
By assumption, $d_S(a, b) \le K$, so the geodesic we add has length at most $K$.
Hence, by Claim~\ref{claim: ladder increase}, we add at most $2K\xi(S)-2$ to the length of the sequence of multicurves, so the length of the sequence of multicurves corresponding to~$\W{1}$ is at most $(2K+2)\xi(S)=T_{\xi(S)}$.

Now, assume for induction that $k_i \le L_i$ and that the length of the sequence of curves corresponding to $\W{k_i}$ is at most $T_i$.
To get from $\W{k_i}$ to $\W{k_{i-1}}$, we successively add tight ladders in bricks whose base surfaces are witnesses of complexity $i$.
There are $\N{k_i}{i-1}$ of these bricks, where $\N{k_i}{i-1}$ is the $(i-1)$-th term of the $\ksep$\hyp{}complexity.
Hence, $k_{i-1}-k_i=\N{k_i}{i-1}$.
Moreover, by Claim \ref{claim: length to bricks}, $\N{k_i}{i-1} \le \frac{3}{2} T_i +\xi(S) \le 2 T_i$.
Hence, $k_{i-1} \le k_i + 2T_i \le L_i +2T_i = L_{i-1}$. 

Between $\W{k}$ and $\W{k+1}$, for $k_i \le k \le k_{i-1}-1$, we add at most $2(K+2Mk)(i-1)$ to the length of the sequence of multicurves, from Claim \ref{claim: bounded projection} and Claim \ref{claim: ladder increase}.
The length of the sequence of multicurves determined by $\W{k_{i-1}}$ is hence at most:

\begin{align*}
&T_i + 2(K+2Mk_i)(i-1) + 2(K+2M(k_i+1))(i-1) + \dots \\
&\qquad\qquad\qquad\qquad\qquad\qquad + 2(K+2M(k_i+\N{k_i}{i-1}-1))(i-1)\\
=&T_i + 2\N{k_i}{i-1}(K+2Mk_i)(i-1) +4M(0+1+ \dots + (\N{k_i}{i-1}-1))(i-1)\\
=&T_i + 2\N{k_i}{i-1}(K+2Mk_i)(i-1) +2M(\N{k_i}{i-1}-1)\N{k_i}{i-1}(i-1)\\
\le&T_i + 2\cdot2T_i(K+2ML_i)(i-1) +2M\cdot2T_i\cdot2T_i(i-1)\\
\le&T_i + 4T_i(K+2ML_i)\xi(S) +8MT_i^2\xi(S)=T_{i-1}.
\end{align*}

\qedclaim

In particular, the length of the sequence of multicurves determined by $\W{k_1}$ is at most $T_1$, which is a function of $K$ and $\xi(S)$.
Since the $\ksep$\hyp{}complexity at this stage is $(0, 0, \dots, 0)$, this sequence of multicurves in fact gives a path in $\ksep_{\mathcal{G}}(S)$ joining $a$ and $b$.
Taking $K'=T_1$, this completes the proof of Proposition~\ref{proposition: ksep bound}, and hence also of Theorem~\ref{theorem: ksep hhs}.

\section{The quasi\hyp{}isometry} \label{section: sep}

We now relate $\ksep_\G(S)$ to $\G(S)$ to prove Theorem \ref{theorem: g hhs}.
Since every vertex of $\G(S)$ is also a vertex of $\ksep_\G(S)$, there is a natural inclusion $\phi\colon \G(S) \rightarrow \ksep_\G(S)$.
Again, we are considering $\G(S)$ and $\ksep_\G(S)$ as discrete sets of vertices with the induced combinatorial metric.
Moreover, we again assume all curves are pairwise in minimal position.

\begin{proposition} \label{proposition: sep qi ksep}
Let $S$ be a surface and $\G(S)$ a \mygraph.
Then the inclusion $\phi\colon \G(S) \rightarrow \ksep_\G(S)$ is a quasi\hyp{}isometry.
\end{proposition}

\begin{lemma} \label{lemma: g to k lipschitz}
The inclusion $\phi\colon \G(S) \rightarrow \ksep_\G(S)$ is Lipschitz.
\end{lemma}

\begin{proof}
Suppose that $a$ and $b$ are adjacent vertices of~$\G(S)$.
We want to bound the distance between their images in~$\ksep_\G(S)$.
By assumption, there is an upper bound on the intersection number of any pair of adjacent vertices of~$\G(S)$.
Hence, up to the action of~$\mcg(S)$, there are only finitely many such pairs.
Moreover, since $\mcg(S)$ acts isometrically on $\ksep_\G(S)$, these finitely many pairs give all possible distances in $\ksep_\G(S)$ between adjacent vertices of~$\G(S)$.
As observed in Claim~\ref{claim: ksep connected}, $\ksep_\G(S)$ is connected, so there is a maximal distance $D$ in $\ksep_\G(S)$ between the vertices in any such pair.
Hence, if $a$ and $b$ are adjacent in~$\G(S)$, then $d_{\ksep_\G(S)}(\phi(a), \phi(b)) \le D$, and by the triangle inequality, $\phi$~is $D$\hyp{}Lipschitz.
\end{proof}

\begin{claim} \label{claim: n exists}
There exists $N$ such that for any vertex $a$ of $\ksep_\G(S)$, there is some vertex $b$ of $\G(S)$ satisfying $i(a, b) \le N$.
\end{claim}

Up to the action of~$\mcg(S)$, there are finitely many vertices of~$\ksep_\G(S)$.
Choose any vertex $c$ of~$\G(S)$, and let $N$ be the maximal number of times $c$ intersects any of the vertices in a (finite) list of representatives of $\mcg(S)$\hyp{}orbits in~$\ksep_\G(S)$.

\qedclaim

\begin{lemma} \label{lemma: coarse lipschitz retract}
Let $N$ be as in Claim~\ref{claim: n exists}.
Given a vertex~$a$ of~$\ksep_\G(S)$, define $C_a = \lbrace b \in \G(S) \mid i(a, b) \le N \rbrace$.
There exists $N'$ depending only on $N$ and $S$ such that $\diam_{\G(S)}(C_a) \le N'$.
\end{lemma}

\begin{proof}
Let $Y_1, \dots, Y_k$ be the components of $S \sminus a$, where $k$ is bounded above by~$-\chi(S)$.
Suppose that $b$ and $b'$ are two elements of~$C_a$.
We will show that there exists $N'$ as required which is an upper bound on their distance in~$\G(S)$, not depending on $b$ and~$b'$.
Note that for any connected union $Y$ of components of $S \sminus a$, the intersection $b \cap Y$ is a collection of at most $N$ arcs and at most $\xi(Y)$ simple closed curves in~$Y$.
Although there are technically uncountably many ways to place the endpoints of the arcs, combinatorially there are finitely many possibilities for $b \cap Y$ up to $\mcg(Y)$, depending on~$N$.

By the definition of $\ksep_\G(S)$, each $Y_i$ is not a witness for $\G(S)$.
In particular, either $Y_i$ is a copy of $S_{0,3}$, or $\xi(Y_i)\ge1$ and there exists a vertex $c$ of $\G(S)$ which does not intersect $Y_i$.
Moreover, in the latter case, since there are finitely many possibilities for $b \cap (S \sminus Y_i)$ up to $\mcg(Y_i)$ (with given endpoints in $\partial_S Y_i$), we can choose $c$ so that $i(c, b)$ is bounded in terms of $N$ and $\xi(S)$.

We proceed to successively adjust $b$ in each of the $Y_i$ to bring it close to $b'$.
Note that in order to give an upper bound on the distance in $\G(S)$ between two vertices, it is sufficient to bound their intersection number, since  $\G(S)$ is connected and up to the action of $\mcg(S)$ there are only finitely many pairs of vertices of $\G(S)$ with a given intersection number.

Consider the intersection of $b$ and $b'$ with~$Y_1$.
Since $b$ and $b'$ each intersect $a$ at most $N$ times, there are finitely many possibilities, up to $\mcg(Y_1)$, for each of these.
In particular, we can choose representatives from their respective $\mcg(Y_1)$~orbits with intersection number bounded above by a constant depending on $N$ and $\xi(Y_1) \le \xi(S)$.

\emph{Case 1: $\xi(Y_1)\ge1$.} We can apply a mapping class $\bar{f}_1$ on $Y_1$ to move $b \cap Y_1$ to have bounded intersection with $b' \cap Y_1$.
We can then extend $\bar{f}_1$ to a mapping class $f_1$ on $S$ by applying the identity on $S \sminus Y_1$, to obtain a new vertex $f_1(b)$ of~$\G(S)$.
Since $Y_1$ is not a witness for $\G(S)$ and is not a pair of pants, there exists a vertex $c$ of $\G(S)$ contained in $S \sminus Y_1$ (possibly peripheral in this subsurface).
Moreover, by the discussion above, we can assume that $c$ has bounded intersection with~$b$, where the bound depends only on $N$ and~$\xi(S)$. 
Since $b$ and $f_1(b)$ coincide in $S \sminus Y_1$, $c$ has bounded intersection with both of these multicurves, and hence we have a bound on the distance from $b$ to $f_1(b)$, via~$c$, depending only on $N$ and~$\xi(S)$.

\emph{Case 2: $\xi(Y_1)=0$.} Here $Y_1$ is a copy of $S_{0, 3}$, and $b \cap Y_1$ and $b' \cap Y_1$ are each collections of at most $N$ arcs.
If we allow isotopies to preserve the boundary setwise rather than pointwise, then any two isotopy classes of arcs on $S_{0, 3}$ intersect at most twice.
Hence, up to twists on the boundary, the number of intersections between $b$ and $b'$ inside $Y_1$ is bounded in terms of $N$.
However, we also need to deal with these twists.

Let the boundary components of $Y_1$ be $\gamma_1$, $\gamma_2$, $\gamma_3$, and take an annular neighbourhood $A_i$ of each $\gamma_i$ inside $Y_1$.
We can assume, by isotoping intersections into these neighbourhoods, that in the complement of $A_1 \cup A_2 \cup A_3$ in~$Y_1$, any pair of arcs from $b$ and $b'$ intersect at most twice.
Now, we can apply a power $T_{\gamma_1}^{n_1}$ of a Dehn twist on $A_1$ so that $T_{\gamma_1}^{n_1}(b) \cap A_1$ has bounded intersection with $b' \cap A_1$.
Since $A_1$ is not a witness for $\G(S)$, there exists a vertex $c_1$ of $\G(S)$ which has trivial subsurface projection to $\C(A_1)$.
Then $c_1$ can be isotoped into $S \sminus A_1$ (where it may be peripheral).
As above, we can choose $c_1$ to have bounded intersection with $b$ and $T_{\gamma_1}^{n_1}(b)$.

We can now repeat this for $A_2$ and $A_3$, obtaining a sequence $b$, $c_1$, $T_{\gamma_1}^{n_1}(b)$, $c_2$, $T_{\gamma_2}^{n_2}T_{\gamma_1}^{n_1}(b)$, $c_3$, $T_{\gamma_3}^{n_3}T_{\gamma_2}^{n_2}T_{\gamma_1}^{n_1}(b)$, where each consecutive pair of vertices has uniformly bounded intersection.
Hence this gives us a bound on the distance in $\G(S)$ between $b$ and $f_1(b)=T_{\gamma_3}^{n_3}T_{\gamma_2}^{n_2}T_{\gamma_1}^{n_1}(b)$.


We now move on to $Y_2$.
We have the same two cases as before.
If $\xi(Y_2)\ge1$, then we apply a mapping class on $Y_2$ moving $f_1(b) \cap Y_2 = b \cap Y_2$ to have bounded intersection with $b' \cap Y_2$, and extend to a mapping class $f_2$ on~$S$.
We then find a vertex of $\G(S)$ contained in $S \sminus Y_2$ which has bounded intersection with $f_1(b)$ and~$b'$.
Note that since $i(f_1(b), a) \le i(b, a)$ (and every $Y_i$ is a subsurface of the same~$S$) the upper bounds on intersection numbers remain the same as for~$Y_1$.
If $\xi(Y_2)=0$, then we use the same procedure as above to remove twists about boundary components of~$Y_2$, and we again have a mapping class~$f_2$, which is now a product of Dehn twists.


We continue to repeat this procedure for every $Y_i$.
We will then have a sequence $g_0(b), g_1(b), \dots, g_k(b)$ of vertices of $\G(S)$, where $g_0=\text{id}$ and $g_i=f_i \circ f_{i-1} \circ \dots \circ f_1$ for $i \ge 1$.
There exists an upper bound, depending only on $N$ and $\xi(S)$, on the distance in $\G(S)$ between $g_{i-1}(b)$ and $g_i(b)$, for $1 \le i \le k$.
Moreover, $i(g_k(b), b')$ is bounded in terms of $N$ and $\xi(S)$ so we have an upper bound on the distance between $b$ and $b'$ as required.
\end{proof}

\begin{proof}[Proof of Proposition \ref{proposition: sep qi ksep}]
Firstly, the map $\phi\colon \G(S) \rightarrow \ksep_\G(S)$ is coarsely surjective as follows.
By Claim \ref{claim: n exists}, there exists $N$ such that for every vertex $a$ of $\ksep_\G(S)$, there is some vertex $b$ of $\G(S)$ satisfying $i(a, b) \le N$.
Moreover, since there are finitely many vertices of $\ksep_\G(S)$ up to $\mcg(S)$, and the property of being a vertex of $\G(S)$ is $\mcg(S)$\hyp{}invariant, this implies that there exists $R$ such that every vertex $a$ of $\ksep_\G(S)$ is at distance at most $R$ from some vertex of$~\G(S)$.

By Lemma \ref{lemma: g to k lipschitz}, the map $\phi$ is Lipschitz.

We claim that the map associating to each vertex $a$ of $\ksep_\G(S)$ the set $C_a$ defined in Lemma \ref{lemma: coarse lipschitz retract} is a coarse Lipschitz retract for $\phi$.
We have shown in Lemma \ref{lemma: coarse lipschitz retract} that there is an upper bound $N'$ on the diameter of any $C_a$, depending only on $N$ and~$\xi(S)$.
It remains to prove that if $a$ and $b$ are adjacent in $\ksep_\G(S)$, then $C_a$ and $C_b$ are close, that is, there is an upper bound on $\diam(C_a \cup C_b)$.
Up to the action of $\mcg(S)$, there are finitely many pairs of adjacent vertices $a$, $b$ of $\ksep_\G(S)$ (for example, since $i(a, b) \le 2$).
Notice, moreover, that if $f$ is a mapping class on $S$, then $C_{f(a)}=f(C_a)$.
Hence, it is also true that up to $\mcg(S)$ there are finitely many pairs $C_a$, $C_b$ where $a$ and $b$ are adjacent in~$\ksep_\G(S)$.
Since each $C_a$ has diameter at most~$N'$, each $C_a \cup C_b$ also has finite diameter.
Taking a representative for each $\mcg(S)$ orbit of pairs $C_a$, $C_b$, we can take the maximum of the finite list of finite diameters, giving an upper bound $N''$ on the diameter of $C_a \cup C_b$ for adjacent $a$ and~$b$.
Hence, given any path $a_0, \dots, a_n$ in $\ksep_\G(S)$, we may choose a representative from each $C_{a_i}$, for $0 \le i \le n$, to obtain a sequence of vertices of $\G(S)$ joining $a_0$ and $a_n$ where the distance between consecutive terms is at most~$N''$.
In particular, if $a_0=\phi(b)$ and $a_n=\phi(b')$ then we can choose $b$ from $C_{a_0}$ and $b'$ from $C_{a_n}$, giving the required lower bound on $d_{\ksep_\G(S)}(\phi(b), \phi(b'))$ in terms of $d_{\G(S)}(b, b')$.
\end{proof}

\begin{proof}[Proof of Theorem~\ref{theorem: g hhs}]
This is almost immediate from Theorem~\ref{theorem: ksep hhs} and Proposition~\ref{proposition: sep qi ksep} as follows.

As observed in \cite[Section~1.4]{hhs1}, if $(\Lambda, \mathfrak{S})$ is a hierarchically hyperbolic space with respect to projections $\pi_X\colon \Lambda \rightarrow \C(X)$ for $X \in \mathfrak{S}$, and if $f\colon \Lambda' \rightarrow \Lambda$ is a quasi\hyp{}isometry, then $(\Lambda', \mathfrak{S})$ is a hierarchically hyperbolic space.
More specifically, the hierarchically hyperbolic structure on $\Lambda$ has the same hyperbolic spaces $\C(X)$ for $X \in \mathfrak{S}$, and the projection maps are the compositions \mbox{$\pi_X \circ f$}.

Here, the quasi\hyp{}isometry is the inclusion $\phi\colon \G(S) \rightarrow \ksep_\G(S)$.
The projection maps $\pi_X \circ \phi\colon \G(S) \rightarrow \C(X)$ are still the ordinary subsurface projections, and the hierarchically hyperbolic structure for~$\G(S)$ is essentially the same as for~$\ksep_\G(S)$.
\end{proof}

\appendix

\section{A hierarchically hyperbolic structure on the arc graph} \label{appendix: arc graph hhs}

In this appendix, we construct a quasi\hyp{}isometry between the arc graph of a surface with boundary and a \mygraph{} with the same set of witnesses, showing that the arc graph has a hierarchically hyperbolic structure with respect to subsurface projections to its witnesses (this projection is defined by surgeries in the same way as for curves).
The fact that the arc graph has a distance formula in terms of subsurface projections to witnesses was already proved by Masur and Schleimer in~\cite{ms}.
They prove moreover that the arc graph is Gromov hyperbolic, so that it also has a trivial hierarchically hyperbolic structure with repect to the identity map to itself.

\begin{definition}
Let $S$ be a surface with boundary and $\Delta$ a union of (at least one but not necessarily all) boundary components of $S$.
The arc graph with respect to $\Delta$, $\A(S, \Delta)$, has:
\begin{itemize}
\item a vertex for every properly embedded arc in $S$ with both endpoints in $\Delta$, up to isotopies which allow endpoints to move inside the boundary,
\item an edge between two distinct vertices if they have disjoint representatives.
\end{itemize}
\end{definition}

We will assume that $S$ and $\Delta$ are always such that $\A(S, \Delta)$ is connected.
The witnesses for $\A(S, \Delta)$ are precisely those essential subsurfaces which contain $\Delta$ and are not homeomorphic to $S_{0,3}$ (note that we do not consider a peripheral annulus to be an essential subsurface).

\begin{theorem} \label{theorem: arc graph hhs}
The graph $\A(S, \Delta)$ is a hierarchically hyperbolic space with respect to subsurface projections to its witnesses, whenever it is connected.
\end{theorem}

\begin{definition}
The graph $\G(S, \Delta)$ has a vertex for:
\begin{enumerate}
\item any curve cutting off a copy of $S_{0,3}$ at least one of whose other boundary components is in $\Delta$, 
\item any pair of disjoint curves cutting off a copy of $S_{0,3}$ whose other boundary component is in $\Delta$.
\end{enumerate}
There is an edge between vertices $a$ and $b$ if $i(a, b) \le 4$. 
\end{definition}

\begin{figure}[h!]
\centering
\begin{subfigure}[b]{0.25\textwidth}
\centering
\includegraphics[width=\textwidth]{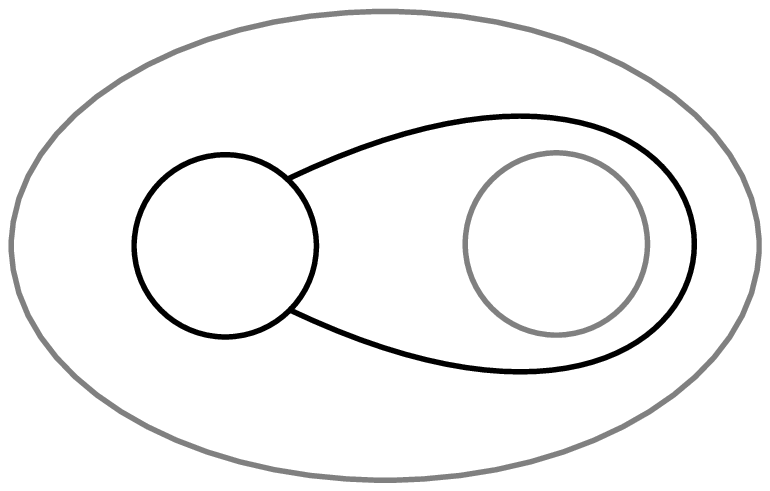}
\subcaption{}
\label{figure: psiexample1}
\end{subfigure}
\quad
\begin{subfigure}[b]{0.25\textwidth}
\centering
\includegraphics[width=\textwidth]{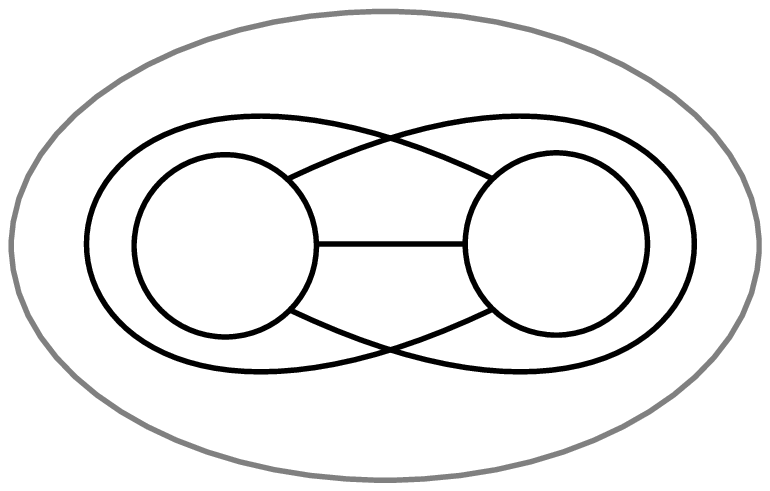}
\subcaption{}
\label{figure: psiexample2}
\end{subfigure}
\quad
\begin{subfigure}[b]{0.25\textwidth}
\centering
\includegraphics[width=\textwidth]{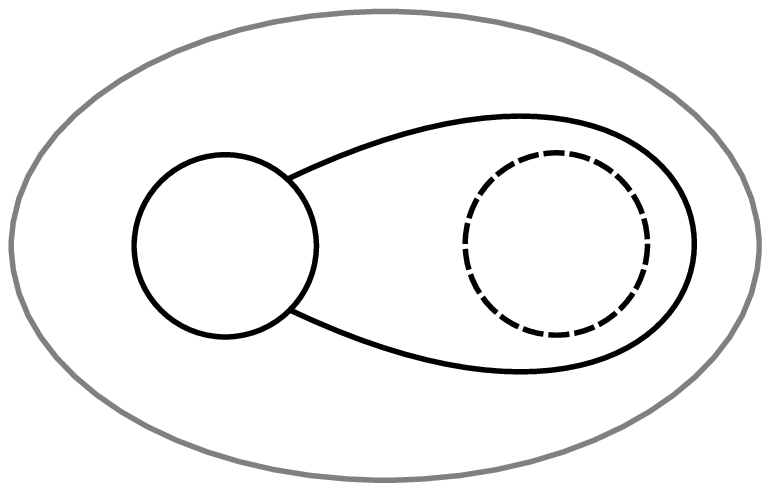}
\subcaption{}
\label{figure: psiexample3}
\end{subfigure}
\caption{Vertices of $\G(S,\Delta)$, in grey. Solid black curves represent components of $\partial S$ in $\Delta$ and dotted curves are components of $\partial S$ not in~$\Delta$. The arcs are the pre\hyp{}images of each vertex under the map $\psi$ in Proposition~\ref{proposition: arc graph qi}.}
\label{figure: psiexamples}
\end{figure}

\begin{claim} \label{claim: arc graph witnesses}
The sets of witnesses for $\A(S, \Delta)$ and for $\G(S, \Delta)$ are the same.
\end{claim}

We want to show that, as for $\A(S, \Delta)$, the set $\X$ of witnesses for $\G(S, \Delta)$ is equal to the set $\mathfrak{Y}$ of essential subsurfaces of positive complexity which contain~$\Delta$.
\begin{itemize}
\item[$\X \subseteq \mathfrak{Y}$] Suppose that $X$ is a witness for $\G(S, \Delta)$ and that some component $\delta$ of $\Delta$ is not contained in~$X$.
Let $Z$ be the component of $S \sminus X$ containing~$\delta$.
Since $X$ is an essential subsurface, $Z$ is not an annulus, so there exists an essential subsurface $W \subset Z$ where $W \cong S_{0,3}$ and $W$ contains~$\delta$.
At least one boundary component of $W$ must be a non\hyp{}peripheral curve in~$S$.
These non\hyp{}peripheral boundary components define a vertex of $\G(S, \Delta)$ which does not intersect $X$, contradicting that $X$ is a witness.
 \item[$\mathfrak{Y} \subseteq \X$] Let $X$ be a subsurface in $\mathfrak{Y}$ and suppose that there exists a vertex $a$ of $\G(S, \Delta)$ which does not intersect~$X$.
The vertex $a$ is a curve or pair of curves bounding an essential subsurface $W \cong S_{0,3}$, such that at least one boundary component of $W$ is in~$\Delta$.
Since $a$ does not intersect $X$, we have that $X$ is contained either in $W$ or in $S \sminus W$.
The former is impossible since subsurfaces in $\mathfrak{Y}$ are assumed to have positive complexity.
The latter contradicts that $X$ contains all components of~$\Delta$.
\end{itemize}

\qedclaim

In particular, Claim~\ref{claim: arc graph witnesses} shows that $\G(S, \Delta)$ has no annular witnesses.
The other conditions for $\G(S, \Delta)$ to be a \mygraph{} are easily verified, with the exception perhaps of connectedness.
However, since $\A(S, \Delta)$ is connected, Proposition~\ref{proposition: arc graph qi} will prove that $\G(S, \Delta)$ is also connected.

For an arc $\alpha$ in $\A(S, \Delta)$, define $\psi(\alpha)$ to be the non\hyp{}peripheral boundary components of a small regular neighbourhood of $\alpha \cup \partial S$ (see Figure~\ref{figure: psiexamples}).
In other words, $\psi(\alpha)$ is the multicurve obtained by the subsurface projection of $\alpha$ to~$\C(S)$.
Note that $\psi(\alpha)$ is a vertex of $\G(S, \Delta)$.

\begin{proposition} \label{proposition: arc graph qi}
The map $\psi\colon \A(S, \Delta) \rightarrow \G(S, \Delta)$ defined above is a quasi\hyp{}isometry.
Moreover, the subsurface projection from $\A(S, \Delta)$ to the curve graph of a witness $X$ agrees with the subsurface projection from $\G(S, \Delta)$ to $\C(X)$ precomposed with $\psi$, up to uniformly bounded error.
\end{proposition}

As in the proof of Theorem~\ref{theorem: g hhs}, Proposition~\ref{proposition: arc graph qi}, along with the fact that $\G(S, \Delta)$ is a \mygraph{}, will complete the proof of Theorem \ref{theorem: arc graph hhs}.

\begin{proof}[Proof of Proposition \ref{proposition: arc graph qi}]
We will first prove that $\psi$ is Lipschitz and then construct a Lipschitz quasi\hyp{}inverse.
Suppose that arcs $\alpha$ and $\beta$ are at distance~1 in $\A(S, \Delta)$.
Then $\alpha$ and $\beta$ are disjoint.
Hence $i(\psi(\alpha), \psi(\beta)) \le 4$, so $\psi(\alpha)$ and $\psi(\beta)$ are adjacent in $\G(S, \Delta)$.
By the triangle inequality, $\psi$ is 1\hyp{}Lipschitz.

Now, define a map $\eta\colon \G(S, \Delta) \rightarrow \A(S, \Delta)$ as follows.
For a vertex $a$ of $\G(S, \Delta)$, let $X_a$ be the component of $S \sminus a$ which is a copy of $S_{0,3}$ (in the few cases where $S \sminus a$ could be two copies of $S_{0,3}$, each with boundary components in~$\Delta$, then we choose one).
If only one boundary component of $X_a$ is in $\Delta$, define $\eta(a)$ to be the (unique) arc in $X_a$ which has both endpoints in this boundary component.
If two boundary components of $X_a$ are in $\Delta$, define $\eta(a)$ to be the arc in $X_a$ with one endpoint in each of these two components.

The map $\psi\circ\eta\colon \G(S, \Delta) \rightarrow \G(S, \Delta)$ is equal to the identity map on $\G(S, \Delta)$.
Consider the map $\eta\circ\psi\colon \A(S,\Delta) \rightarrow \A(S, \Delta)$.
Assume for now that only one component of $S \sminus \psi(\alpha)$ is a copy of~$S_{0,3}$.
First suppose the arc $\alpha$ is such that $\psi(\alpha)$ bounds a pair of pants with only one boundary component in~$\Delta$.
Then $\alpha$ can only be an arc with both endpoints in this boundary component, and hence exactly coincides with $\eta\circ\psi(\alpha)$.
Suppose now that $\psi(\alpha)$ bounds a pair of pants with two boundary components in $\Delta$.
Then there are three possibilities for $\alpha$ - one with both endpoints in one of the boundary components, one with both endpoints in the other, and one with an endpoint in each.
The third of these possibilities is exactly $\eta\circ\psi(\alpha)$, and is disjoint from the other two possibilities, so that $d_{\A(S, \Delta)}(\alpha, \eta\circ\psi(\alpha)) \le 1$.
If $\psi(\alpha)$ has two components which are copies of $S_{0,3}$ and which have boundary components in $\Delta$, then we have the additional possibility that when defining $\eta\circ\psi(\alpha)$ we choose the component which did not contain~$\alpha$.
Then $\eta\circ\psi(\alpha)$ is disjoint from~$\alpha$.
Hence, in general, $\eta\circ\psi(\alpha)$ is at distance at most 1 from $\alpha$ in $\A(S, \Delta)$.


We now claim that $\eta$ is Lipschitz.
Let $a$ and $b$ be adjacent vertices of $\G(S, \Delta)$.
We will find a bound on the intersection number of $\eta(a)$ and $\eta(b)$.
Let $X_a \cong S_{0,3}$ be the component of $X \sminus a$ containing $\eta(a)$, and similarly for $X_b$.
Consider $a \cap X_b$.
Since $\partial_S X_b = b$ and $i(a,b)\le4$, this consists of one or two arcs (each arc corresponds to two  intersections with~$b$, since $b$ separates~$S$).
Since $\eta(b)$ is contained in $X_b$, so are any intersections of $\eta(b)$ with $a$.
Up to twists on the boundary, any two arcs on a pair of pants intersect at most twice.
Moreover, the endpoints of $\eta(b)$ must be on different boundary components of $X_b$ from the endpoints of $a \cap X_b$.
Hence, $i(a, \eta(b))\le4$.
Now, $a=\partial_S X_a$, so $\eta(b) \cap X_a$ consists of at most two arcs.
By the same argument as above, $i(\eta(a), \eta(b))\le4$ (now arcs of $\eta(a)$ and $\eta(b)$ could have endpoints on the same boundary component of $X_a$, but this would have to be in $\partial S$, so that we are allowed to move the endpoints to remove any boundary twists).
Hence $\eta$ is Lipschitz, completing the proof that $\psi$ is a quasi\hyp{}isometry.


It remains to prove that, given a witness $X$ and an arc $\alpha$ in $\A(S, \Delta)$, the subsurface projection of $\alpha$ to $\C(X)$ uniformly coarsely coincides with the subsurface projection of $\psi(\alpha)$.
Note that since $X$ is a witness for $\A(S, \Delta)$ and $\G(S, \Delta)$, both $\alpha$ and $\psi(\alpha)$ do indeed have non\hyp{}trivial subsurface projection to~$X$.
To obtain the subsurface projection by surgery (see Section~\ref{subsection: subsurface projection}), we require the arcs and curves we are projecting to be in minimal projection with $\partial_S X$.
It is possible to realise $\alpha$, $\psi(\alpha)$ and $\partial_S X$ simultaneously in pairwise minimal position; for example we may choose a hyperbolic metric on $S$ and take geodesic representatives.
From the definition of the map~$\psi$, it is clear that $\alpha$ and $\psi(\alpha)$ are disjoint in minimal position.
Hence the respective arcs of intersection of $\alpha$ and $\psi(\alpha)$ with $X$ are also pairwise disjoint.
By the surgery arguments given in \cite[Lemma~2.2]{mm2}, we see that $d_X(\alpha, \psi(\alpha))\le2$.
\end{proof}

\end{document}